\documentclass{amsart}

\usepackage{amsmath}
\usepackage{amsthm}
\usepackage{amssymb}

\usepackage{mathrsfs}
\usepackage[all]{xy}

\newcommand{\set}[2]{\left\{ #1 \mid #2 \right\}}

\newcommand{\adjunction}[4]{\xymatrix{ #1 \ar@<1ex>[rr]^-{#3} && #2 \ar@<1ex>[ll]^-{#4}}}
\newcommand{\HH}{\operatorname{H}}
\newcommand{\tensor}{\otimes}

\newcommand{\htensor}{\hat{\otimes}}
\newcommand{\Hom}{\operatorname{Hom}}
\newcommand{\ZZ}{\mathbb{Z}}
\newcommand{\QQ}{\mathbb{Q}}
\newcommand{\CC}{\mathbb{C}}

\newcommand{\aut}{\operatorname{aut}}

\newcommand{\Map}{\operatorname{Map}}

\newcommand{\gl}{\mathfrak{g}}
\newcommand{\hll}{\mathfrak{h}}
\newcommand{\hh}{\mathfrak{h}}
\newcommand{\MC}{\mathsf{MC}}

\newcommand{\Lie}{Lie}
\newcommand{\CP}[1]{\CC \hspace{-2.5pt} \operatorname{P}^{#1}}
\newcommand{\Tor}{\operatorname{Tor}}
\newcommand{\MCmod}{\mathscr{MC}}
\newcommand{\kk}{\Bbbk}

\newcommand{\limo}{{\varprojlim}^1}

\newtheorem{theorem}{Theorem}[section]

\newtheorem{proposition}[theorem]{Proposition}
\newtheorem{corollary}[theorem]{Corollary}
\newtheorem{lemma}[theorem]{Lemma}

\theoremstyle{definition}
\newtheorem{definition}[theorem]{Definition}
\newtheorem{remark}[theorem]{Remark}

\title[Rational homotopy theory of mapping spaces]{Rational homotopy theory of mapping spaces via Lie theory for $L_\infty$-algebras}
\author{Alexander Berglund}
\address{Department of Mathematics, Stockholm University, SE-106 91 Stockholm, Sweden}
\email{alexb@math.su.se}

\begin{document}

\begin{abstract}
We calculate the higher homotopy groups of the Deligne-Getzler $\infty$-groupoid associated to a nilpotent $L_\infty$-algebra. As an application, we present a new approach to the rational homotopy theory of mapping spaces.
\end{abstract}

\maketitle

\section{Introduction}
In \cite{Getzler} Getzler associates an $\infty$-groupoid $\gamma_\bullet(\gl)$ to a nilpotent $L_\infty$-algebra $\gl$, which generalizes the Deligne groupoid of a nilpotent differential graded Lie algebra \cite{Getzler2,Getzler,GM,Hinich}. It is well known that the set of path components $\pi_0\gamma_\bullet(\gl)$ may be identified with the `moduli space' $\MCmod(\gl)$ of equivalence classes of Maurer-Cartan elements, which plays an important role in deformation theory. Our main result is the calculation of the higher homotopy groups.

\begin{theorem} \label{thm:main1}
For a nilpotent $L_\infty$-algebra $\gl$ there is an explicitly defined natural group isomorphism
\begin{equation*} \label{eq:iso}
B\colon \HH_n(\gl^\tau) \rightarrow \pi_{n+1}(\gamma_\bullet(\gl),\tau),\quad n\geq 0,
\end{equation*}
where $\gl^\tau$ denotes the $L_\infty$-algebra $\gl$ twisted by the Maurer-Cartan element $\tau$. The group structure on $\HH_0(\gl^\tau)$ is given by the Campbell-Hausdorff formula.
\end{theorem}

As a corollary, we obtain a characterization of when an $L_\infty$-morphism induces an equivalence between the associated $\infty$-groupoids.
\begin{corollary}
An $L_\infty$-morphism between nilpotent $L_\infty$-algebras $f\colon \gl\rightarrow \hll$ induces an equivalence of $\infty$-groupoids $\gamma_\bullet(\gl)\rightarrow \gamma_\bullet(\hll)$ if and only if the following two conditions are satisfied:
\begin{enumerate}
\item The map on Maurer-Cartan moduli $\MCmod(\gl)\rightarrow\MCmod(\hll)$ is a bijection.
\item The $L_\infty$-morphism $f^\tau\colon \gl^\tau\rightarrow \hll^{f_*(\tau)}$ induces an isomorphism in homology in non-negative degrees, for every Maurer-Cartan element $\tau$ in $\gl$.
\end{enumerate}
\end{corollary}

When the ground field is $\QQ$, it follows from Theorem \ref{thm:main1} that the connected components of $\gamma_\bullet(\gl)$ are $\QQ$-local spaces, so by Sullivan's rational homotopy theory \cite{Sullivan}, their homotopy types are modeled by commutative differential graded algebras. We have the following explicit description of Sullivan models for the components.

\begin{corollary} \label{cor:main2}
Let $\gl$ be a nilpotent $L_\infty$-algebra. The component of $\gamma_\bullet(\gl)$ that contains a given Maurer-Cartan element $\tau$ is homotopy equivalent to $\gamma_\bullet(\gl_{\geq 0}^\tau)$, where $\gl_{\geq 0}^\tau$ denotes the truncation of $\gl^\tau$. Hence, provided $\gl$ is of finite type, the Chevalley-Eilenberg construction $C^*(\gl_{\geq 0}^\tau)$ is a Sullivan model for the component of $\gamma_\bullet(\gl)$ that contains $\tau$.
\end{corollary}

Furthermore, as will be important for our applications, we extend the domain of definition of the Deligne-Getzler $\infty$-groupoid to \emph{complete} $L_\infty$-algebras (by which we mean, essentially, inverse limits of towers of nilpotent $L_\infty$-algebras), and we show that our main results generalize to this setting.

We remark that the inclusion of $\gamma_\bullet(\gl)$ into the nerve $\MC_\bullet(\gl)$ is a homotopy equivalence (see \S\ref{sec:getzler}), so our results could as well have been stated for the nerve. For the proofs, however, Getzler's formulas for the horn fillers in $\gamma_\bullet(\gl)$ will play a role.

While Theorem \ref{thm:main1} and its corollaries should have independent interest, the application that originally motivated us is within the rational homotopy theory of mapping spaces. The key to this application is the following observation.
\begin{theorem} \label{thm:mainbs}
Let $X$ be a connected space, let $Y$ be a nilpotent space of finite $\QQ$-type and let $Y_\QQ$ be its $\QQ$-localization. If $A$ is a commutative differential graded algebra model for $X$ and $L$ is an $L_\infty$-algebra model for $Y$, then there is a homotopy equivalence
$$\Map(X,Y_\QQ) \simeq \gamma_\bullet(A\hat{\tensor} L),$$
where $A\hat{\tensor} L$ is the inverse limit of the nilpotent $L_\infty$-algebras $A\tensor L/L_{\geq r}$.
\end{theorem}
This appears as Theorem \ref{thm:map tensor} below. The following is a direct consequence of Theorem \ref{thm:main1}, Corollary \ref{cor:main2} and Theorem \ref{thm:mainbs}.

\begin{theorem} \label{thm:main3}
Under the hypotheses of Theorem \ref{thm:mainbs}, there is bijection
$$[X,Y_\QQ] \cong \MCmod(A\hat{\tensor} L).$$
Moreover, given a map $f\colon X\rightarrow Y_\QQ$ whose homotopy class corresponds to the equivalence class of the Maurer-Cartan element $\tau\in A\hat{\tensor} L$, there are isomorphisms
$$\pi_{n+1}(\Map(X,Y_\QQ),f) \cong \HH_n (A\hat{\tensor} L^\tau),\quad n\geq 0.$$
For $n\geq 1$ this is an isomorphism of rational vector spaces. For $n = 0$ it is an isomorphism of groups where the group structure on the right hand side is given by the Campbell-Hausdorff formula. If $A\hat{\tensor} L$ is of finite type, then the Chevalley-Eilenberg construction $C^*((A\hat{\tensor} L)_{\geq 0}^\tau)$ is a Sullivan model for the connected component of the mapping space $\Map(X,Y_\QQ)$ that contains $f$.
\end{theorem}

Rational models for mapping spaces have been studied before by many authors, see for instance \cite{BL,BPS,BS2,BFM-tran,BFM-jap,FLS,Haefliger,LS} and the recent survey \cite[\S3.1]{Smith}. The approach presented here has the advantage that one gets explicit and computable $L_\infty$-algebra models for the connected components that are expressed in terms of natural constructions on models for the source and target --- tensoring, twisting by a Maurer-Cartan element, and truncating. We have included a number of examples in the last section to illustrate this point. During the preparation of this paper, we learned that variants of Theorem \ref{thm:main3} have been obtained independently by Buijs-F\'elix-Murillo \cite{BFM-new} and Lazarev \cite[\S8]{Lazarev}, but using different methods. The main point with our approach is that Theorem \ref{thm:main3} is essentially a corollary of our general results about Deligne-Getzler $\infty$-groupoids. This suggests that Deligne-Getzler $\infty$-groupoids should have a greater role to play in rational homotopy theory.

\section{$L_\infty$-algebras} \label{sec:L-infinity}
This section contains basic definitions and facts about $L_\infty$-algebras.

\subsection*{Graded vector spaces}
We work over a ground field $\kk$ of characteristic zero. A graded vector space is a collection $V =\{V_i\}_{i\in \ZZ}$ of vector spaces over $\kk$. We say that $V$ is of \emph{finite type} if each component $V_i$ is finite dimensional. We say that it is \emph{bounded below (above)} if $V_i = 0$ for $i\ll 0$ ($i\gg0$), \emph{bounded} if it is bounded below and above, and \emph{finite dimensional} if it is bounded and of finite type. We use the convention that $V^i = V_{-i}$, and think of upper degrees as cohomological and lower degrees as homological. The graded vector space $\Hom(V,W)$ is defined by $\Hom(V,W)_i = \prod_{p+q = i} \Hom(V^p,W_q)$. The dual graded vector space is defined by $V^\vee = \Hom(V,\kk)$, where $\kk$ is viewed as a graded vector space concentrated in degree $0$. The \emph{$n$-fold suspension} $V[n]$ is defined by $V[n]_i = V_{i-n}$.

\subsection*{$L_\infty$-algebras}
An $L_\infty$-algebra is a graded vector space $\gl$ together with maps
$$\ell_r\colon \gl^{\tensor r} \rightarrow \gl,\quad x_1\tensor \ldots \tensor x_r \mapsto [x_1,\ldots,x_r],$$
of degree $r-2$, for all $r\geq 1$, satisfying the following axioms:
\begin{itemize}
\item (Anti-symmetry) $[\ldots,x,y,\ldots] = -(-1)^{|x||y|} [\ldots,y,x,\ldots]$.
\item (Generalized Jacobi identities) For every $n\geq 1$ and all $x_1,\ldots, x_n\in \gl$,
$$\sum_{p=1}^n \sum_{\sigma} (-1)^{\epsilon} [[x_{\sigma_1},\ldots,x_{\sigma_p}],x_{\sigma_{p+1}},\ldots,x_{\sigma_n}] = 0,$$
where the inner sum is over all permutations $\sigma$ of $\{1,\ldots,n\}$ such that $\sigma_1<\ldots < \sigma_p$ and $\sigma_{p+1}<\ldots < \sigma_n$. The sign is given by
$$\epsilon = p + \sum_{\substack{i<j, \\ \sigma_i>\sigma_j}} (|x_i||x_j|+1).$$
\end{itemize}
We will write $\delta(x) = [x]$. For $n=1,2,3$, the generalized Jacobi identities are equivalent to
\begin{align*}
\delta^2(x) & = 0, \\
\delta[x,y] & = [\delta x,y] + (-1)^{|x|}[x,\delta y], \\
[x,[y,z]] - [[x,y],z] - (-1)^{|x||y|}[y,[x,z]] & = \pm(\delta \ell_3 + \ell_3\delta)(x\tensor y\tensor z).
\end{align*}
In particular, $\gl$ has an underlying chain complex,
$$\xymatrix{\cdots \ar[r] & \gl_1 \ar[r]^{\delta_1} & \gl_0 \ar[r]^{\delta_0} & \gl_{-1} \ar[r] & \cdots}$$
The binary bracket $\ell_2$ satisfies the Jacobi identity up to homotopy, the ternary bracket $\ell_3$ being a contracting homotopy. For this reason an $L_\infty$-algebra may be thought of as a Lie algebra `up to homotopy'.

An $L_\infty$-algebra concentrated in degree $0$ is the same thing as an ordinary Lie algebra, because in this case $\ell_r = 0$ for $r\ne 2$ for degree reasons, and the generalized Jacobi identities reduce to the classical Jacobi identity. An $L_\infty$-algebra $\gl$ is called \emph{abelian} if $\ell_r = 0$ for $r\geq 2$. An abelian $L_\infty$-algebra is the same thing as a chain complex $(\gl,\delta)$. An $L_\infty$-algebra is called \emph{minimal} if its differential $\delta = \ell_1$ is zero. An $L_\infty$-algebra with $\ell_r = 0$ for $r\geq 3$ is the same thing as a differential graded Lie algebra $(\gl,\delta,[,])$.

\subsection*{The Chevalley-Eilenberg construction}
There is an alternative and more compact definition of $L_\infty$-algebras. An $L_\infty$-structure on a graded vector space $\gl$ is the same thing as a coderivation differential $d$ on the symmetric coalgebra $\Lambda_c \gl[1]$. The relation between $d$ and the brackets is given by the following: write $d = d_1+d_2 +\ldots$, where $d_r$ is the component of $d$ that decreases word-length by $r-1$. Each $d_r$ is a coderivation and is therefore determined by its restriction $d_r\colon \Lambda_c^r \gl[1] \rightarrow \gl[1]$. The formula
$$d_r(sx_1\ldots sx_r) =  \epsilon s[x_1,\ldots,x_r],$$
where $\epsilon = |x_{r-1}| + |x_{r-3}| + \cdots$, defines  $d$ in terms of the higher brackets and vice versa. The condition $d^2 = 0$ is equivalent to the generalized Jacobi identities for $\gl$. The \emph{bar construction} of an $L_\infty$-algebra $\gl$ is defined to be the cocommutative differential graded coalgebra
$$\mathscr{C}_*(\gl) = (\Lambda_c \gl[1],d).$$
A convenient feature of the compact definition of an $L_\infty$-algebra is that an \emph{$L_\infty$-morphism} $f\colon \gl \to \hh$ may be defined simply as a morphism of cocommutative differential graded coalgebras $f\colon \mathscr{C}_*(\gl) \to \mathscr{C}_*(\hh)$. When written out, an $L_\infty$-morphism $f\colon \gl\to \hh$ corresponds to a collection of maps,
$$f_n\colon \gl^{\tensor n} \to \hh,\quad n\geq 1,$$
of degree $n-1$ satisfying certain compatibility conditions. In particular, $f_1\colon \gl\to \hh$ is a morphism of chain complexes. We say that $f$ is a quasi-isomorphism if $f_1$ is. If $\gl$ is a differential graded Lie algebra, then $\mathscr{C}_*(\gl)$ coincides with the classical construction due to Quillen \cite[Appendix B]{Quillen}. The \emph{Chevalley-Eilenberg construction} $C^*(\gl)$ is by definition the dual commutative differential graded algebra $C^*(\gl) = \mathscr{C}_*(\gl)^\vee$. If $L$ is of finite type and non-negatively graded, then $C^*(L)$ may be identified with a commutative differential graded algebra (cdga) of the form $(\Lambda V,d)$, where $\Lambda V$ denotes the free graded commutative algebra on the graded vector space $V = L[1]^\vee$. In fact, the Chevalley-Eilenberg construction gives a one-to-one correspondence
$$
\left\{ \parbox{4.2cm}{\center Finite type $L_\infty$-algebras $L$ with $L=L_{\geq 0}$ \endcenter} \right\} \longleftrightarrow \left\{ \parbox{4cm}{\center Finite type cdgas $(\Lambda V,d)$ with $V = V^{\geq 1}$ \endcenter} \right\}.
$$

\subsection*{Nilpotent $L_\infty$-algebras}
Let $\gl$ be an $L_\infty$-algebra. We define a \emph{compatible filtration} to be a descending filtration, $\gl = F^1\gl \supseteq F^2 \gl \supseteq \cdots$, such that
\begin{equation}
[F^{i_1}\gl,\ldots, F^{i_r}\gl] \subseteq F^{i_1+\cdots + i_r} \gl
\end{equation}
for all $r\geq 1$ and all $i_1,\ldots,i_r$. Every $\gl$ admits at least one compatible filtration, for instance the constant one with $F^i\gl = \gl$ for all $i$. Moreover, the level-wise intersection of any family of compatible filtrations is again compatible. It follows that there exists a smallest compatible filtration, namely the intersection of all compatible filtrations. We denote this by
$$\gl = \Gamma^1 \gl \supseteq \Gamma^2 \gl \supseteq \cdots$$
and call it the \emph{lower central series} of $\gl$. Concretely, $\Gamma^k\gl$ is spanned by all possible bracket expressions one can form using at least $k$ elements from $\gl$. If $\gl$ is an ordinary Lie algebra, then $\Gamma^k\gl$ is the ordinary lower central series.

\begin{definition}
Let $\gl$ be an $L_\infty$-algebra. We say that $\gl$ is
\begin{enumerate}
\item \emph{nilpotent} if $\Gamma^k \gl = 0$ for some $k$,
\item \emph{degree-wise nilpotent} if for every $n$ there is a $k$ such that $(\Gamma^k\gl)_n = 0$.
\end{enumerate}
In other words, $\gl$ is (degree-wise) nilpotent if the lower central series eventually terminates at $0$ (degree-wise).
\end{definition}

Evidently, every nilpotent $L_\infty$-algebra is degree-wise nilpotent. The converse is not true. In fact, every positively graded $L_\infty$-algebra is automatically degree-wise nilpotent, but of course not nilpotent in general. The reader may check that a non-negatively graded Lie algebra $L$ is degree-wise nilpotent if and only if $L_0$ is a nilpotent Lie algebra and $L_n$ is a nilpotent $L_0$-module for every $n$. In Theorem \ref{thm:sullivan} we will see that non-negatively graded degree-wise nilpotent $L_\infty$-algebras of finite type correspond exactly to Sullivan models for nilpotent topological spaces of finite $\QQ$-type. Thus, the notion of degree-wise nilpotence is closely related to the notion of nilpotence for topological spaces. 

\begin{remark}
We warn the reader that the `lower central series' defined in \cite{Getzler} is not a filtration: it is defined by $F^1\gl = \gl$ and, for $i >1$,
$$F^i \gl = \sum_{i_1+\cdots + i_k=i} [F^{i_1}\gl,\ldots,F^{i_k}\gl].$$
For instance, if $\gl$ has trivial binary bracket but non-trivial ternary bracket, then $F^2\gl \not\supseteq F^3\gl$. However, the definition of nilpotence given in \cite{Getzler} (stating that $\gl$ is nilpotent if $F^i \gl = 0$ for all $i\gg 0$) makes sense and agrees with our definition.
\end{remark}

\subsection*{Degree-wise nilpotent $L_\infty$-algebras and Sullivan algebras}
Recall \cite[\S12]{FHT-RHT} that a \emph{Sullivan algebra} is a cdga of the form $(\Lambda V,d)$, where $V$ is a graded vector space that is concentrated in positive cohomological degrees and that admits a filtration,
$$0 = V(-1) \subseteq V(0)\subseteq V(1) \subseteq \ldots \subseteq V,\quad V = \bigcup V(k),$$
such that $dV(k) \subseteq \Lambda V(k-1)$ for all $k$. It is called \emph{minimal} if the differential is decomposable in the sense that $d(V)\subseteq \Lambda^+ V \cdot \Lambda^+ V$.

\begin{theorem} \label{thm:sullivan}
The Chevalley-Eilenberg construction gives a bijection

$$
\left\{ \parbox{5cm}{\center Degree-wise nilpotent, finite type $L_\infty$-algebras $L$ with $L=L_{\geq 0}$ \endcenter} \right\} \longleftrightarrow \left\{ \parbox{4.4cm}{\center Finite type Sullivan algebras $(\Lambda V,d)$ with $V = V^{\geq 1}$ \endcenter} \right\}.
$$
Minimal $L_\infty$-algebras correspond to minimal Sullivan algebras.
\end{theorem}

\begin{proof}
We may identify $C^*(L)=(\Lambda V,d)$ where $V = L[1]^\vee$. Clearly, $V = V^{\geq 1}$ if and only if $L = L_{\geq 0}$. Suppose given a filtration $V(k)$ of $V$ that exhibits $(\Lambda V,d)$ as a Sullivan algebra. Since $V$ is of finite type, we may without loss of generality assume that each $V(k)$ is finite dimensional (if this is not the case, then we can work with the filtration $W(k) = V(0)^{\leq k} + V(1)^{\leq k-2} + V(2)^{\leq k-4} + \ldots$ instead.)
If we let $F^kL \subseteq L = V[1]^\vee$ be the annihilator of $V(k)[1]$, then we get a decreasing filtration $L = F^{-1}L \supseteq F^0L \supseteq F^1 L \supseteq \ldots$. The condition $dV(k+1) \subseteq \Lambda V(k)$ translates into the condition
\begin{equation} \label{eq:complete}
\ell_n(F^kL,L,\ldots,L)\subseteq F^{k+1}L,\quad   \mbox{for all $k$ and all $n\geq 1$.}
\end{equation}
In particular, $F^k L$ is an $L_\infty$-ideal. Each quotient $L/F^kL$ is finite dimensional because $V(k)$ is assumed to be finite dimensional. This implies that $(L/F^kL)_{>p} =0$ for some $p$. Since $L = L_{\geq 0}$, it follows that all brackets of arity $>p+2$ vanish in $L/F^kL$. This together with \eqref{eq:complete} implies that the $L_\infty$-algebra $L/F^kL$ is nilpotent. In other words, for every $k$ there is an $r_k$ such that $\Gamma^{r_k}L \subseteq F^kL$. Since $V$ is the union of all $V(k)$ and $V$ is of finite type, we have that for every $p$ there is a $k$ such that $V(k)[1]^p = V[1]^p$, or equivalently $(F^k L)_p = 0$. Hence, $(\Gamma^{r_k}L)_p = 0$ for this $k$, which shows that $L$ is degree-wise nilpotent. 

Conversely, assume that $L$ is degree-wise nilpotent. Define a descending filtration $F^kL$ by setting
$$F^{2i-1}L = \Gamma^i L,\quad F^{2i}L = \Gamma^i L \cap d^{-1}(\Gamma^{i+1}L).$$
Then \eqref{eq:complete} is satisfied. (Note that \eqref{eq:complete} could be violated for $n=1$ if one were to take $F^kL = \Gamma^k L$.) If we define $V(k) \subseteq V = L[1]^\vee$ to be the annihilator of $F^kL[1]$, then the condition \eqref{eq:complete} makes sure that $dV(k) \subseteq \Lambda V(k-1)$ for all $k$. The filtration $V(k)$ exhausts $V$ because $L$ is degree-wise nilpotent.
\end{proof}

\subsection*{Maurer-Cartan elements}
If $\gl$ is a degree-wise nilpotent $L_\infty$-algebra the following is a finite sum for every $\tau\in\gl_{-1}$:
\begin{equation} \label{eq:mc}
\mathcal{F}(\tau) = \sum_{k\geq 1} \frac{1}{k!} [\tau^{\wedge k}].
\end{equation}
Here we write $[\tau^{\wedge k}]$ for $[\tau,\ldots,\tau]$ ($k$ copies of $\tau$). If $\mathcal{F}(\tau) = 0$ then $\tau$ is called a \emph{Maurer-Cartan element}, and we let $\MC(\gl)$ denote the set of all such elements. Given $\tau\in\MC(\gl)$ the formula
$$[\alpha_1,\ldots,\alpha_r]_\tau = \sum_{k \geq 0} \frac{1}{k!} [\tau^{\wedge k},\alpha_1,\ldots,\alpha_r],\quad r\geq 1,$$
defines a new $L_\infty$-algebra structure on the underlying graded vector space of $\gl$ \cite[Proposition 4.4]{Getzler}. Denote this $L_\infty$-algebra by $\gl^\tau$.

\subsection*{Extension of scalars}
Given a commutative differential graded algebra (cdga) $A$ and an $L_\infty$-algebra $\gl$, we can extend scalars and form a new $L_\infty$-algebra $A \tensor \gl$ with differential and brackets defined by
$$\delta(x\tensor \alpha) = d_A(x)\tensor \alpha + (-1)^{|x|} x\tensor d_{\gl}(\alpha),$$
$$[x_1\tensor \alpha_1,\ldots,x_r\tensor \alpha_r] = (-1)^{\sum_{i<j} |\alpha_i||x_j|} x_1\ldots x_r\tensor [\alpha_1,\ldots,\alpha_r],\quad r\geq 2.$$
Similarly, we may endow $\gl \tensor A$ with an $L_\infty$-algebra structure.

\section{Getzler's $\infty$-groupoid $\gamma_\bullet(\gl)$} \label{sec:getzler}
In this section we will review some basic properties of the nerve $\MC_\bullet(\gl)$ and the Deligne-Getzler $\infty$-groupoid $\gamma_\bullet(\gl)$ associated to a nilpotent $L_\infty$-algebra $\gl$.

Let $\Omega_\bullet$ denote the simplicial cdga with $n$-simplices
$$\Omega_n = \frac{\kk[t_0,\ldots,t_n]\tensor \Lambda(dt_0,\ldots,dt_n)}{(\sum_i t_i -1, \sum_i dt_i)},\quad |t_i|=0, |dt_i| = 1.$$
For a non-decreasing map $\varphi\colon [n] \rightarrow [m]$, the morphism of cdgas $\varphi^*\colon \Omega_m\rightarrow \Omega_n$ is determined by the formula $\varphi^*(t_i) = \sum_{j\in \varphi^{-1}(i)} t_j$. Note that $\Omega_n$ is not finite dimensional, but it is bounded. Indeed, $\Omega_n^i = 0$ unless $0\leq i\leq n$.

\begin{definition} (\cite{Getzler,Hinich})
The \emph{nerve} of a nilpotent $L_\infty$-algebra $\gl$ is the simplicial set
$$\MC_\bullet(\gl) = \MC(\gl \tensor \Omega_\bullet).$$
\end{definition}

Getzler defines $\gamma_\bullet(\gl)$ as a certain subcomplex of $\MC_\bullet(\gl)$. To state the definition we need to introduce some notation. For each $n$, the \emph{elementary differential forms},
$$\omega_{i_0,\ldots,i_k} = k!\sum_{j=0}^k (-1)^j t_{i_j} dt_{i_0}\ldots \widehat{dt_{i_j}} \ldots dt_{i_k}\in \Omega_n^k,\quad 0\leq i_0, \ldots , i_k\leq n,$$
span a finite dimensional subcomplex $C_n\subseteq \Omega_n$. The cochain complex $C_n$ is isomorphic to the normalized cochain complex of the standard $n$-simplex. These assemble into an inclusion of simplicial cochain complexes $C_\bullet \subseteq \Omega_\bullet$. By \cite{Dupont,Getzler}, there is a projection $P_\bullet\colon \Omega_\bullet\rightarrow C_\bullet$ and a `gauge' $s_\bullet\colon \Omega_\bullet^* \rightarrow \Omega_\bullet^{*-1}$ such that
$$1-P_\bullet = ds_\bullet +s_\bullet d,\quad s_\bullet^2 = 0.$$
\begin{definition}[Getzler \cite{Getzler}]
Let $\gl$ be a nilpotent $L_\infty$-algebra. The simplicial set $\gamma_\bullet(\gl)$ is defined by
$$\gamma_\bullet(\gl)  = \set{\alpha\in \MC_\bullet(\gl)}{s_\bullet \alpha = 0} \subseteq \MC_\bullet(\gl).$$
\end{definition}

According to \cite[Corollary 5.11]{Getzler}, the inclusion of $\gamma_\bullet(\gl)$ into $\MC_\bullet(\gl)$ is a homotopy equivalence, so these simplicial sets are indistinguishable from the point of view of homotopy theory. However, $\gamma_\bullet(\gl)$ has some remarkable additional features. In Lie theory one associates a group $G$ to a nilpotent Lie algebra $\gl$ by way of the \emph{Campbell-Hausdorff formula}: as sets $G=\gl$, and the multiplication is given by
$$\alpha \cdot \beta = \log(e^\alpha e^\beta)  = \alpha + \beta + \frac{1}{2}[\alpha,\beta] + \frac{1}{12}[\alpha,[\alpha,\beta]] + \frac{1}{12}[[\alpha,\beta],\beta] + \cdots,$$
see, e.g., \cite{Serre}. It turns out that $\gamma_\bullet(\gl)$ is isomorphic to the nerve of the group $G$ in this situation. For general nilpotent $L_\infty$-algebras, $\gamma_\bullet(\gl)$ is an $\infty$-groupoid, in particular a Kan complex, and Getzler's iterative formulas for the horn fillers may be interpreted as generalized Campbell-Hausdorff formulas. The construction may also be viewed as a non-abelian version of the Dold-Kan construction (see \cite[Proposition 5.1]{Getzler}).

\section{The homotopy groups of $\gamma_\bullet(\gl)$} \label{sec:gamma}
This section contains the proof of Theorem \ref{thm:main1} and is the core of the paper. For $\gl$ a nilpotent $L_\infty$-algebra, we will define the natural map
$$B_n^\tau\colon \HH_n(\gl^\tau)\rightarrow \pi_{n+1}(\gamma_\bullet(\gl),\tau),\quad n\geq 0,$$
and prove that it is an isomorphism. We first deal with the case $\tau = 0$. Then in Proposition \ref{prop:stranslate} we show how to reduce everything to this case. Finally, we explain how to extend the results to degree-wise nilpotent $L_\infty$-algebras.

As mentioned in the previous section, $\gamma_\bullet(\gl)$ is a Kan complex. Recall that elements $[\alpha]$ of the $n$th homotopy group $\pi_n(X,v)$ of a Kan complex $X$, based at a vertex $v\in X_0$, are represented by $n$-simplices $\alpha\in X_n$ such that $\partial_i \alpha = v$ for all $i$ (where we interpret $v$ as a degenerate $(n-1)$-simplex). Two $n$-simplices $\alpha$ and $\beta$ represent the same homotopy class if and only if there is an $(n+1)$-simplex $\omega$ such that $\partial\omega = (v,\ldots,v,\alpha,\beta)$.  Here we use the notation $\partial\omega = (\partial_0\omega,\ldots,\partial_{n+1}\omega)$. For $n\geq 1$, the group structure on $\pi_n(X,v)$ is given by $[\alpha]\cdot [\beta] = [\partial_n \omega]$, where $\omega$ is any $(n+1)$-simplex with boundary $\partial\omega = (v,\ldots,v,\alpha,-,\beta)$, the existence of which is guaranteed by the Kan condition. More generally, for every $(n+1)$-simplex $\omega$ such that $\partial_j\partial_i \omega = v$ for all $i,j$, we have the homotopy addition theorem, which says that
$$[\partial_0 \omega] - [\partial_1 \omega] + [\partial_2\omega] - \cdots + (-1)^{n+1} [\partial_{n+1} \omega] \in \pi_n(X,v),\quad n\geq 2,$$
$$[\partial_0 \omega] [\partial_2 \omega] = [\partial_1 \omega] \in\pi_1(X,v),\quad n=1,$$
see \cite[Theorem III.3.13]{GJ}.

\subsection*{Definition and fundamental properties of the map $B_n$}
To simplify notation, let $\pi_n(\gl) = \pi_n(\gamma_\bullet(\gl),0)$ for $n\geq 1$ and $\pi_0(\gl) = \pi_0(\gamma_\bullet(\gl))$. Let
$$\omega_{i_0,\ldots,i_k}^n = k!\sum_{j=0}^k (-1)^j t_{i_j} dt_{i_0}\ldots \widehat{dt_{i_j}} \ldots dt_{i_k}\in \Omega_n^k.$$
To simplify notation further we will write
\begin{align*}
\omega^n & = \omega_{0\ldots n}^n = n!dt_1\ldots dt_n, \\
\omega_{\widehat{r}}^n & = \omega_{0\ldots \widehat{r}\ldots n}^n,
\end{align*}
and we will drop the superscript when it is clear from the context. We note that the simplicial faces of $\omega_{\widehat{r}}$ are given by $\partial_i(\omega_{\widehat{r}}) = \omega$ if $i = r$ and zero otherwise. Furthermore, with respect to the de Rham differential, we have that $d\omega_{\widehat{r}} = (-1)^r \omega$.

\begin{definition}
Let $\gl$ be a nilpotent $L_\infty$-algebra. Define
$$B_n\colon \HH_n(\gl)\rightarrow \pi_{n+1}(\gl),\quad B_n [\alpha] = [\alpha\tensor \omega^{n+1}].$$
\end{definition}

\begin{proposition}
The map $B_n\colon \HH_n(\gl)\rightarrow \pi_{n+1}(\gl)$ is well-defined for all $n\geq 0$.
\end{proposition}

\begin{proof}
First of all, $\alpha\tensor \omega$ is a Maurer-Cartan element: Since $\delta \alpha = 0$ and $d\omega = 0$ we have $(d+\delta)(\alpha\tensor \omega) = 0$. Since $\omega \omega = 0$ we have $[(\alpha\tensor \omega)^{\wedge \ell}] = 0$ for all $\ell \geq 2$.

Suppose that $[\alpha] = [\beta] \in \HH_n(\gl)$, say $\delta \chi = \alpha -\beta$ where $\chi \in \gl_{n+1}$. Then consider the degree $-1$ element
\begin{align*}
\lambda_0 & : = (d+\delta)(\chi\tensor \omega_{\hat{2}} - (-1)^n \alpha\tensor \omega_{\hat{0}\hat{2}} - (-1)^n \beta \omega_{\hat{1}\hat{2}})) \\
& = \alpha \tensor \omega_{\hat{0}} +\beta\tensor \omega_{\hat{1}} - (-1)^n \chi \tensor \omega.
\end{align*}
The simplicial boundary of $\lambda_0$ is
\begin{equation} \label{eq:boundary}
\partial \lambda_0 = (\alpha\tensor \omega,\beta\tensor \omega,0,\ldots,0).
\end{equation}
It is obvious that $(d+\delta)\lambda_0 = 0$ and that $[\lambda_0^{\wedge \ell}] = 0$ for all $\ell\geq 3$. If $n>0$ or one of $\alpha$ or $\beta$ is zero, then also $[\lambda_0,\lambda_0] = 0$, so that $\lambda_0\in\gamma_{n+2}(\gl)$ in these cases. Then \eqref{eq:boundary} shows that $[\alpha\tensor \omega] = [\beta\tensor \omega] \in \pi_{n+1}(\gl)$. In the remaining case $n=0$ and $\alpha,\beta \ne 0$, we have that
\begin{equation} \label{eq:lambda_0}
[\lambda_0,\lambda_0] = -[\alpha,\beta]\tensor t_2\omega_{012},
\end{equation}
so $\lambda_0$ is not necessarily a Maurer-Cartan element in this case. However, using \cite[Lemma 5.3]{Getzler} we can find an element $\lambda \in \gamma_2(\gl)$ with the property that $\epsilon^2 \lambda = 0$, $P_2 R^2 \lambda = \lambda_0$ and
\begin{equation} \label{eq:b2}
\partial \lambda = (\alpha\tensor \omega,\beta\tensor \omega, \partial_2 \lambda).
\end{equation}
We will argue that $[\partial_2 \lambda] = 0\in \pi_1(\gl)$. Then \eqref{eq:b2} shows that $[\alpha\tensor \omega] = [\beta\tensor \omega]\in\pi_1(\gl)$.

The element $\lambda$ is obtained as a limit $\lambda = \lim_{k\rightarrow \infty} \lambda_k$ where the elements $\lambda_k$ are defined iteratively by
\begin{equation} \label{eq:iteration}
\lambda_{k+1} = \lambda_0 - \sum_{\ell\geq 2} \frac{1}{\ell !} F [\lambda_k^{\wedge \ell}],\quad k\geq 0
\end{equation}
where $F = P_2h_2^2 + s_2 \colon \Omega_2^* \rightarrow \Omega_2^{*-1}$.

Observe that since $[\alpha,\beta] = \delta [\chi,\beta]$, \eqref{eq:lambda_0} shows that $[\lambda_0,\lambda_0]$ is a $\delta$-boundary. Moreover, by definition $\lambda_0\in (Z_0(\gl)\tensor\Omega_2^1)\oplus (\gl_1\tensor \Omega_2^2)$. Assume by induction that
\begin{itemize}
\item $[\lambda_k,\lambda_k]$ and $\lambda_k-\lambda_0$ are $\delta$-boundaries.
\item $\lambda_k \in (Z_0(\gl)\tensor\Omega_2^1)\oplus (\gl_1\tensor \Omega_2^2)$.
\end{itemize}
Then it follows that the same is true for $\lambda_{k+1}$. Indeed, for degree reasons $[\lambda_k^{\wedge \ell}] = 0$ for $\ell\geq 3$, so the iterative formula \eqref{eq:iteration} reduces to
$$\lambda_{k+1} = \lambda_0 -\frac{1}{2}F[\lambda_k,\lambda_k].$$
This implies that $\lambda_{k+1} \in (Z_0(\gl)\tensor\Omega_2^1)\oplus (\gl_1\tensor \Omega_2^2)$ and that $\lambda_{k+1}-\lambda_0$ is a $\delta$-boundary. The identity
$$[\lambda_{k+1},\lambda_{k+1}] = [\lambda_0,\lambda_0] + [\lambda_{k+1}-\lambda_0,\lambda_{k+1}+\lambda_0]$$
together with the facts that $\lambda_{k+1}-\lambda_0$ and $[\lambda_0,\lambda_0]$ are $\delta$-boundaries and that $\lambda_k+\lambda_0\in (Z_0(\gl)\tensor\Omega_2^1)\oplus (\gl_1\tensor \Omega_2^2)$ imply that $[\lambda_{k+1},\lambda_{k+1}]$ is a $\delta$-boundary. This finishes the inductive step.

It follows that $\lambda-\lambda_0$ is a $\delta$-boundary. Since $\partial_2 \lambda_0 = 0$ this implies that $\partial_2 \lambda$ is a $\delta$-boundary. But since $\partial_2\lambda\in \gamma_1(\gl)$ and $\partial_2\lambda \in Z_0(\gl)\tensor \Omega_1^1$ this implies that $\partial_2\lambda = \xi\tensor \omega_{01}$ for some $\delta$-boundary $\xi \in \gl_0$. It follows from the first part of the proof that $[\partial_2\lambda] = 0 \in \pi_1(\gl)$.
\end{proof}

\begin{proposition} \label{prop:homo}
The map $B_n\colon \HH_n(\gl)\rightarrow \pi_{n+1}(\gl)$ is a homomorphism of abelian groups for $n\geq 1$.
\end{proposition}

\begin{proof}
Let $\alpha,\beta\in \gl_n$ be two cycles. We need to show that
$$[(\alpha+\beta)\tensor \omega] = [\alpha\tensor \omega] + [\beta\tensor\omega]\in \pi_{n+1}(\gl).$$
This means that we have to find an element $\lambda\in\gamma_{n+2}(\gl)$ with simplicial boundary
$$\partial \lambda = (\alpha\tensor \omega,(\alpha+\beta)\tensor \omega,\beta\tensor \omega,0\ldots,0).$$
We claim that
\begin{align*}
\lambda & = \alpha\tensor\omega_{\hat{0}} + (\alpha + \beta)\tensor \omega_{\hat{1}} + \beta \tensor \omega_{\hat{2}}
\end{align*}
satisfies the requirements. Indeed, $\lambda$ has the correct simplicial boundary. One calculates that $(d+\delta)\lambda =0$, and since $n\geq 1$ we have that $[\lambda^{\wedge \ell}] = 0$ for $\ell \geq 2$ for degree reasons. Therefore $\lambda \in \gamma_{n+2}(\gl)$.
\end{proof}

If $\gl$ is a nilpotent $L_\infty$-algebra, then the zeroth homology $\HH_0(\gl)$ is a nilpotent Lie algebra, and it can be given a group structure via the Campbell-Hausdorff formula.

\begin{proposition} \label{prop:ch}
The map $B_0\colon \HH_0(\gl) \rightarrow \pi_1(\gl)$ is a group homomorphism when $\HH_0(\gl)$ is given the Campbell-Hausdorff group structure.
\end{proposition}

\begin{proof}
Given cycles $\alpha,\beta\in\gl_0$, the product of $[\alpha\tensor \omega]$ and $[\beta \tensor \omega]$ in $\pi_1(\gl)$ is represented by $\partial_1\lambda$¸ where $\lambda \in \gamma_2(\gl)$ has simplicial boundary
\begin{equation} \label{eq:b}
\partial \lambda = (\alpha \tensor \omega,\partial_1\lambda,\beta \tensor \omega).
\end{equation}
As in the proof of Proposition \ref{prop:homo}, consider the element
\begin{align*}
\lambda_0 & = (d+\delta)(\alpha\tensor \omega_2 - \beta \tensor \omega_0) \\
& = \alpha\tensor \omega_{12} + (\alpha + \beta)\tensor \omega_{02} + \beta\tensor \omega_{01}.
\end{align*}
It is not necessarily a Maurer-Cartan element, but by \cite[Lemma 5.3]{Getzler}, there is a unique element $\lambda \in \gamma_2(\gl)$ such that $(d+\delta)h^1 \lambda = \lambda_0$. For degree reasons, the recursive formula \cite[(5-20)]{Getzler} defining $\lambda = \lim_{k\rightarrow \infty} \lambda_k$ simplifies to
$$\lambda_{k+1} = \lambda_0 - \frac{1}{2} (P_2 h_2^1 + s_2)[\lambda_k,\lambda_k].$$
It follows that $\lambda \in \gl_0 \tensor \Omega_2^1$, whence $\partial_1 \lambda \in \gl_0 \tensor \Omega_1^1$. Since we also have $\partial_1\lambda \in \gamma_1(\gl)$, this element must be of the form
$$\partial_1 \lambda = \xi \tensor \omega_{01}$$
for some $\xi \in \gl_0$. The coefficient is determined by $\xi = I_{01}(\partial_1 \lambda) = I_{02}(\lambda)$. One finds that $I_{02}(\lambda)$ is given by the Campbell-Hausdorff formula, cf. \cite[p. 296]{Getzler}.
\end{proof}

\subsection*{Fibration sequences}
Let $0\rightarrow \gl' \stackrel{\mu}{\rightarrow} \gl \stackrel{\epsilon}{\rightarrow} \gl''\rightarrow 0$ be a short exact sequence of nilpotent $L_\infty$-algebras. By \cite[Theorem 5.10]{Getzler} there is an associated fibration sequence
$$\gamma_\bullet(\gl')\rightarrow \gamma_\bullet(\gl)\rightarrow \gamma_\bullet(\gl''),$$
whence a long exact sequence of homotopy groups
$$\cdots \rightarrow \pi_{n+2}(\gl'') \stackrel{\partial^\pi}{\rightarrow} \pi_{n+1}(\gl') \rightarrow \pi_{n+1}(\gl) \rightarrow \pi_{n+1}(\gl'') \rightarrow \cdots.$$
On the other hand, the short exact sequence also induces a long exact sequence in homology
$$\cdots\rightarrow \HH_{n+1}(\gl'') \stackrel{\partial^H}{\rightarrow} \HH_n(\gl')\rightarrow \HH_n(\gl)\rightarrow \HH_n(\gl'') \rightarrow \cdots.$$
\begin{proposition} \label{prop:connecting}
Let $0\rightarrow \gl' \stackrel{\mu}{\rightarrow} \gl \stackrel{\epsilon}{\rightarrow} \gl'' \rightarrow 0$ be a short exact sequence of nilpotent $L_\infty$-algebras. The diagram
\begin{equation} \label{eq:connecting}
\xymatrix{\HH_n(\gl'') \ar[d]^-{(-1)^nB_n} \ar[r]^-{\partial^H} & \HH_{n-1}(\gl') \ar[d]^-{B_{n-1}} \\ \pi_{n+1}(\gl'') \ar[r]^-{\partial^\pi} & \pi_n(\gl')}
\end{equation}
commutes for all $n\geq 1$. If $\gl'$ is abelian, then it commutes for $n=0$ as well.
\end{proposition}

\begin{proof}
For a cycle $\alpha''\in\gl_n''$ the class $\partial^H[\alpha'']$ is represented by any cycle $\beta'\in \gl_{n-1}'$ such that
$$\xymatrix{& \alpha \ar@{|->}[d]^\delta \ar@{|->}[r]^{\epsilon_n} & \alpha'' \ar@{|->}[d]^\delta \\  \beta' \ar@{|->}[r]^{\mu_{n-1}} & \beta \ar@{|->}[r]^{\epsilon_{n-1}} & 0 }$$
If we chase $[\alpha'']\in \HH_n(\gl'')$ around the diagram \eqref{eq:connecting}, we get the following picture:
$$\xymatrix{[\alpha''] \ar@{|->}[d]^-{(-1)^n B_n} \ar@{|->}[r]^-{\partial^H} & [\beta'] \ar@{|->}[d]^-{B_{n-1}} \\ (-1)^n[\alpha''\tensor \omega^{n+1}] \ar@{|-->}[r]_-{\partial^\pi}^-? & [\beta'\tensor \omega^{n}]}$$
To compute $\partial^{\pi}(-1)^n[\alpha''\tensor \omega^{n+1}]$ we need to find a lift $\lambda$ in the following diagram
$$\xymatrix{\Lambda^0[n+1] \ar@{^{(}->}[d] \ar[rr]^-{(-,0,\ldots,0)} && \gamma_\bullet(\gl) \ar[d] \\ \Delta[n+1] \ar[rr]_-{(-1)^n \alpha''\tensor \omega^{n+1}} \ar@{-->}[urr]^-\lambda && \gamma_\bullet(\gl'')}$$
Then $\partial_0\lambda$ is an element in $\gamma_n(\gl')$ that represents $\partial^\pi(-1)^n[\alpha''\tensor \omega^{n+1}]$, cf. \cite[Ch. I.7]{GJ}. We claim that
$$\lambda := (d+\delta)(\alpha\tensor \omega_{\widehat{0}}) = \beta\tensor \omega_{\widehat{0}} + (-1)^n\alpha\tensor \omega^{n+1}$$
can be chosen as a lift. Indeed, $\lambda$ is a Maurer-Cartan element because $(d+\delta)\lambda = 0$ and $[\lambda^{\wedge \ell}] = 0$ for $\ell\geq 2$ and $n\geq 1$ for degree reasons. Furthermore, we clearly have $\epsilon(\lambda) = (-1)^n \alpha''\tensor \omega^{n+1}$ and $\partial_i\lambda = 0$ for $0<i\leq n+1$. It follows that $\partial^\pi (-1)^n[\alpha''\tensor \omega^{n+1}]$ is represented by $\partial_0\lambda = \beta\tensor \omega^n$. Thus, $\partial^\pi B_n = (-1)^n B_n \partial^H$.
\end{proof}

\begin{theorem} \label{thm:main nilpotent}
Let $\gl$ be a nilpotent $L_\infty$-algebra. The map
$$B_n\colon \HH_n(\gl) \rightarrow \pi_{n+1}(\gamma_\bullet (\gl),0)$$
is an isomorphism of groups for all $n\geq 0$, where $\HH_0(\gl)$ is given the Campbell-Hausdorff group structure.
\end{theorem}

\begin{proof}
For abelian $\gl$ this follows from the fact that $\gamma_\bullet(\gl)$ is isomorphic to the Dold-Kan construction on the suspended chain complex $\gl[1]$ \cite[Proposition 5.1]{Getzler}. For nilpotent $\gl$ the claim follows from the five lemma by repeated application of Proposition \ref{prop:connecting} to the exact sequences
$$0\rightarrow \Gamma^k\gl/\Gamma^{k+1}\gl \rightarrow \gl/\Gamma^{k+1}\gl \rightarrow \gl/\Gamma^k \gl\rightarrow 0.$$
Note that in order for the five lemma to go through for $n=0$, it is crucial for $B_0$ to be a homomorphism with respect to a group structure where $0$ is the identity element. It is also important that $\Gamma^k\gl/\Gamma^{k+1}\gl$ is abelian so that $B_{-1}$ is defined.
\end{proof}

\begin{remark}
The verifications of the basic properties of the map $B_n$ are surprisingly technical. The referee has suggested a possible alternative approach, whereby one observes that there is a natural quasi-isomorphism of abelian $L_\infty$-algebras,
$$\beta_n\colon \gl[-n-1] \to \mathbf{\Omega}^{n+1}\gl,\quad \alpha\mapsto \alpha\tensor \omega^{n+1},$$
where $\mathbf{\Omega}^{n+1}\gl$ consists of all $\chi\in \gl\tensor \Omega_{n+1}$ with simplicial boundary $\partial \chi = 0$. The referee suggests that there should exist a weak equivalence
$$\upsilon\colon \MC_\bullet(\mathbf{\Omega}^{n+1} \gl) \sim \mathbf{\Omega}^{n+1}\MC_\bullet(\gl)$$
such that the map $B_n$ factors as
$$
\xymatrix{
\HH_n(\gl) \cong \pi_0 \MC_\bullet(\gl[-n-1]) \ar[r]_-{(\beta_n)_*}^-\cong & \pi_0 \MC_\bullet(\mathbf{\Omega}^{n+1}\gl) \ar[r]_-{\upsilon_*}^-\cong & \pi_0 \mathbf{\Omega}^{n+1} \MC_\bullet(\gl) \cong \pi_{n+1}(\gl).
}
$$
We invite the interested reader to investigate this alternative approach.
\end{remark}

\subsection*{The components of $\gamma_\bullet(\gl)$}
So far, we have only been concerned with the homotopy groups of $\gamma_\bullet(\gl)$ at the base-point $0$. The argument given above can be adapted to work for an arbitrary base-point $\tau\in \MC(\gl)$. Alternatively, the following will reduce everything to the base-point $0$. Recall the definition of the twisted $L_\infty$-algebra $\gl^\tau$ from \S\ref{sec:L-infinity}.

\begin{lemma} \label{lemma:translate}
Let $\gl$ be a nilpotent $L_\infty$-algebra and let $\tau\in\MC(\gl)$. Then
$$\MC(\gl^\tau) = \set{\sigma\in\gl_{-1}}{\sigma+\tau \in \MC(\gl)}.$$
\end{lemma}

\begin{proof}
After writing them out, the conditions $\sigma\in\MC(\gl^\tau)$ and $\sigma+\tau\in\MC(\gl)$ both turn out to be equivalent to
$$\sum_{n,k\geq 0} \frac{1}{n!k!} [\sigma^{\wedge n},\tau^{\wedge k}] = 0,$$
where an empty bracket is interpreted as zero.
\end{proof}

If $v$ is a vertex of a Kan complex $X$, then we let $X_v$ denote the simplicial subset consisting of the simplices all of whose vertices are $v$. In other words, an $n$-simplex $x$ belongs to $X_v$ if and only if $\partial_1^{n-i}\partial_0^i x = v$ for all $0\leq i\leq n$. It is easy to see that the simplicial set $X_v$ is a Kan complex. Moreover, it is \emph{reduced} in the sense that it has only one vertex. The inclusion $X_v\subseteq X$ induces a homotopy equivalence between $X_v$ and the connected component of $X$ that contains $v$.

\begin{proposition} \label{prop:stranslate}
Let $\gl$ be a nilpotent $L_\infty$-algebra. For every Maurer-Cartan element $\tau$ in $\gl$ there is an isomorphism of reduced Kan complexes
$$\gamma_\bullet(\gl)_\tau \cong \gamma_\bullet(\gl^\tau)_0.$$
\end{proposition}

\begin{proof}
Since, evidently, $\gl^\tau \tensor \Omega_n = (\gl \tensor \Omega_n)^{\tau\tensor 1}$ it follows from Lemma \ref{lemma:translate} that $n$-simplices $x$ of $\gamma_\bullet(\gl^\tau)_0$ correspond bijectively to $n$-simplices $\tau\tensor 1 + x$ of $\gamma_\bullet(\gl)_\tau$, and it is obvious that this bijection respects face and degeneracy maps.
\end{proof}

\begin{theorem} \label{thm:gammamain}
Let $\gl$ be a nilpotent $L_\infty$-algebra. For every Maurer-Cartan element $\tau$ in $\gl$ and every $n\geq 0$, the map
$$B_n^\tau \colon \HH_n(\gl^\tau)\rightarrow \pi_{n+1}(\gamma_\bullet(\gl),\tau),\quad B_n^\tau[\alpha] = [\tau\tensor 1 + \alpha\tensor \omega^{n+1}],$$
is an isomorphism of groups, where $\HH_0(\gl^\tau)$ is given a group structure via the Campbell-Hausdorff formula.
\end{theorem}

\begin{proof}
This follows by combining Theorem \ref{thm:main} and Proposition \ref{prop:stranslate}.
\end{proof}

Define the \emph{truncation} $\gl_{\geq 0}$ of an $L_\infty$-algebra $\gl$ to be the chain complex
$$\cdots \rightarrow \gl_{2} \stackrel{\delta_{2}}{\rightarrow} \gl_{1} \stackrel{\delta_{1}}{\rightarrow} \ker(\delta_0) \rightarrow 0 \rightarrow \cdots.$$
As the reader may check, the truncation $\gl_{\geq 0}$ is in fact an $L_\infty$-subalgebra of $\gl$.

\begin{corollary} \label{cor:sullivan model}
Let $\gl$ be a nilpotent $L_\infty$-algebra. For every Maurer-Cartan element $\tau$, the inclusion of $(\gl^\tau)_{\geq 0}$ into $\gl^\tau$ induces a homotopy equivalence of reduced Kan complexes
$$\xymatrix{\gamma_\bullet((\gl^\tau)_{\geq 0}) \ar[r]^-\simeq & \gamma_\bullet(\gl)_\tau.}$$
In particular, if $\gl$ is of finite type, then $C^*((\gl^\tau)_{\geq 0})$ is a Sullivan model for the connected component of $\gamma_\bullet(\gl)$ that contains $\tau$.
\end{corollary}

\begin{proof}
By Proposition \ref{prop:stranslate} we may assume that $\tau = 0$ without loss of generality. It is clear that the image of $\gamma_\bullet(\gl_{\geq 0})\rightarrow \gamma_\bullet(\gl)$ is contained in the simplicial subset $\gamma_\bullet(\gl)_0$. We need to show that the resulting map $\gamma_\bullet(\gl_{\geq 0}) \rightarrow \gamma_\bullet(\gl)_0$ is a homotopy equivalence. Since both the source and target are reduced Kan complexes, this happens if and only if the map induces an isomorphism on $\pi_n$ (at the unique base-point) for all $n\geq 1$. Since $\gl_{\geq 0}\rightarrow \gl$ induces an isomorphism in homology in non-negative degrees, Theorem \ref{thm:main} together with naturality of the map $B_n$ finishes the proof:
$$\xymatrix{\HH_n(\gl_{\geq 0}) \ar[d]^-\cong \ar[r]_-{B_n}^-\cong & \pi_{n+1}(\gl_{\geq 0}) \ar[d] \\ \HH_n(\gl) \ar[r]_-{B_n}^-\cong & \pi_{n+1}(\gl).}$$
The last statement follows from Proposition \ref{prop:spatial} below.
\end{proof}

\subsection*{Naturality}
Let us be more explicit about the naturality properties of the map $B_n^\tau$ in Theorem \ref{thm:gammamain}. Given an $L_\infty$-morphism $f\colon \gl \rightarrow \hll$ between nilpotent $L_\infty$-algebras there is an induced map $f_*\colon \MC(\gl)\rightarrow \MC(\hll)$ given by
$$f_*(\tau) = \sum_{n \geq 0} \frac{1}{n!}f_n(\tau^{\wedge n}).$$
Moreover, for each Maurer-Cartan element $\tau$ in $\gl$, there is an induced $L_\infty$-morphism $f^\tau\colon \gl^\tau \rightarrow \hll^{f_*(\tau)}$ given by
$$f_n^\tau(x_1,\ldots,x_n) = \sum_{\ell \geq 0} \frac{1}{\ell!} f_{n+\ell}(\tau^{\wedge \ell},x_1,\ldots,x_n).$$
Naturality means that the following diagram is commutative:
$$\xymatrix{\HH_n(\gl^\tau) \ar[r]^-{B_n^\tau} \ar[d]^-{\HH_n(f^\tau)} & \pi_{n+1}(\gamma_\bullet(\gl),\tau) \ar[d]^-{f_*} \\ \HH_n(\hll^{f_*(\tau)}) \ar[r]^-{B_n^{f_*(\tau)}} & \pi_{n+1}(\gamma_\bullet(\hll),f_*(\tau))}$$
We leave the verification to the reader.

\subsection*{Extension to degree-wise nilpotent $L_\infty$-algebras}
Everything we have done in this section extends verbatim to degree-wise nilpotent $L_\infty$-algebras, because each calculation involves only a finite number of graded pieces at a time. More precisely, if $\gl$ is degree-wise nilpotent, then so is $\gl\tensor \Omega_n$ for every $n$, because $\Omega_n$ is bounded. So we may form the simplicial sets $\MC_\bullet(\gl)$ and $\gamma_\bullet(\gl)$ as before. Note that $\gamma_n(\gl)$ only depends on $\gl_{-2},\gl_{-1},\ldots,\gl_{n-1}$, because $\Omega_n^i = 0$ unless $0\leq i\leq n$. For a fixed $n$, degree-wise nilpotence allows us to choose a $k$ so that $(\Gamma^k\gl)_i = 0$ for $-2\leq i\leq n$. It follows that $\gamma_\bullet(\gl) \rightarrow \gamma_\bullet(\gl/\Gamma^k\gl)$ is an isomorphism on $n$-skeleta. This shows that the homotopy groups of $\gamma_\bullet(\gl)$ may be calculated by passing to suitable nilpotent quotients $\gl/\Gamma^k \gl$. Using this observation, it is easy to see that the proof of Theorem \ref{thm:gammamain} goes through for degree-wise nilpotent $\gl$.

\section{Complete $L_\infty$-algebras}
In this section we will extend the definition of $\gamma_\bullet$ to what we call \emph{complete $L_\infty$-algebras}. This extension will be necessary for our applications to mapping spaces in the next section, in particular when the source is not a finite complex.

\begin{definition}
By a \emph{complete $L_\infty$-algebra} we mean an $L_\infty$-algebra $\gl$ together with a descending filtration of $L_\infty$-ideals, $\gl = F^1\gl \supseteq F^2\gl \supseteq \cdots$, such that
\begin{enumerate}
\item Each quotient $\gl/F^r\gl$ is a nilpotent $L_\infty$-algebra.
\item The canonical map $\gl \rightarrow \varprojlim \gl/F^r\gl$ is an isomorphism. In other words, the topology defined by the filtration is complete and Hausdorff.
\end{enumerate}
Maps of complete $L_\infty$-algebras are required to preserve filtrations.
\end{definition}
Malcev Lie algebras in the sense of Quillen \cite[Appendix A3]{Quillen}, or complete Lie algebras in the sense of Papadima-Suciu \cite[\S 5]{PS}, are examples of complete $L_\infty$-algebras. Other examples are given by degree-wise nilpotent $L_\infty$-algebras; these are complete with respect to the lower central series filtration. For finite type $L_\infty$-algebras, completeness is essentially the same thing as degree-wise nilpotence:

\begin{proposition}
Every complete $L_\infty$-algebra of finite type is degree-wise nilpotent.
\end{proposition}

\begin{proof}
If $\gl$ is complete and of finite type, then the inverse system $\gl/F^r\gl$ must stabilize degree-wise. In other words, for every $n$ there is an $r$ such that $(F^r\gl)_n = 0$. But the $L_\infty$-algebra $\gl/F^r\gl$ is nilpotent, so there is a $k$ such that $\Gamma^k\gl \subseteq F^r\gl$, whence $(\Gamma^k\gl)_n= 0$.
\end{proof}

Of course, the filtration on $\gl$ in the above proof could be different from the lower central series, but it is really only the topology defined by the filtration that matters. The condition $(F^r\gl)_n = 0$ for $r$ large means that the topology on $\gl_n$ is discrete. In effect, it is only for $L_\infty$-algebras that are not of finite type that the notion of completeness goes beyond the notion of degree-wise nilpotence.

We note that the series $d\tau + \frac{1}{2}[\tau,\tau] + \cdots$ converges if $\gl$ is complete, so we may define Maurer-Cartan elements as before.

\begin{definition}
We define the nerve of a complete $L_\infty$-algebra $\gl$ by
$$\MC_\bullet(\gl) = \MC(\gl \htensor \Omega_\bullet),$$
where we take the completed tensor product $\gl \htensor \Omega_\bullet = \varprojlim \big(\gl/F^r\gl \tensor \Omega_\bullet\big)$.
The gauge $s_\bullet$ extends to $\gl\htensor \Omega_\bullet$, and we define $\gamma_\bullet(\gl) = \MC_\bullet(\gl)\cap \ker s_\bullet$ as before.
\end{definition}

This definition of $\MC_\bullet(\gl)$ and $\gamma_\bullet(\gl)$ extends the previous definition. Indeed, if $\gl$ is a degree-wise nilpotent $L_\infty$-algebra equipped with the lower central series filtration, then $\gl \htensor \Omega_\bullet = \gl\tensor \Omega_\bullet$ because every $\Omega_n$ is bounded. As a further justification of this definition, note that if $\gl$ is a Malcev Lie algebra \cite[Appendix A3]{Quillen}, then $\gamma_\bullet(\gl)$ is isomorphic to the nerve of the corresponding Malcev group.

\begin{proposition} \label{prop:mcgamma}
For every complete $L_\infty$-algebra $\gl$, the inclusion of $\gamma_\bullet(\gl)$ into $\MC_\bullet(\gl)$ is a homotopy equivalence. Moreover, the functors $\MC_\bullet$ and $\gamma_\bullet$ take surjective maps of complete $L_\infty$-algebras to Kan fibrations.
\end{proposition}

\begin{proof}
We have that $\MC_\bullet(\gl) = \varprojlim \MC_\bullet(\gl/F^r\gl)$ and $\gamma_\bullet(\gl) = \varprojlim \gamma_\bullet(\gl/F^r\gl)$. The inverse systems are towers of Kan fibrations, so our claims follow from the nilpotent case together with standard facts about towers of fibrations: if a map between towers of fibrations is a levelwise fibration or a weak equivalence then so is the induced map on inverse limits, see, e.g., \cite[Chapter IV]{GJ}.
\end{proof}

\begin{theorem} \label{thm:main}
Let $\gl$ be a complete $L_\infty$-algebra and let $\tau$ be a Maurer-Cartan element in $\gl$. There are natural isomorphism of groups for all $n\geq 0$,
$$B_n^\tau\colon \HH_n(\gl^\tau) \rightarrow \pi_{n+1}(\gamma_\bullet (\gl),\tau),$$
where $\HH_0(\gl^\tau)$ is given a group structure via the Campbell-Hausdorff formula.
\end{theorem}

\begin{proof}
Let $\gl^{(r)}$ denote the nilpotent $L_\infty$-algebra $\gl/F^r\gl$. We define the map $B_n^\tau$ as follows. Given an $n$-cycle $\alpha$ in $\gl$, let $\alpha_r$ and $\tau_r$ denote the images of $\alpha$ and $\tau$ in $\gl^{(r)}$. The elements $\tau_r\tensor 1 + \alpha_r\tensor \omega^{n+1} \in \gl^{(r)}\tensor \Omega_{n+1}$, are compatible and define an $(n+1)$-simplex in $\gamma_\bullet(\gl) = \varprojlim \gamma_\bullet(\gl^{(r)})$, whose simplicial boundary is trivial. We define $B_n^\tau[\alpha]$ to be the class represented by this simplex. To show that $B_n^\tau$ is an isomorphism, it is possible to reduce the case $\tau=0$ as before. In this case, the claim follows from the fact that $B_n\colon \HH_n(\gl^{(r)}) \to \pi_{n+1}(\gl^{(r)},0)$ is an isomorphism for all $r$ (by Theorem \ref{thm:main nilpotent}) together with the fact that $B_n$ is suitably compatible with $\limo$-sequences associated to towers of fibrations. The latter fact will be verified in Proposition \ref{prop:lim^1} below. (The reader may want to skip to the next section at a first reading.)
\end{proof}

By definition, a complete $L_\infty$-algebra $\gl$ is an inverse limit of a tower surjections of nilpotent $L_\infty$-algebras,
$$\cdots \twoheadrightarrow \gl^{(r)} \stackrel{p}{\twoheadrightarrow} \gl^{(r-1)} \twoheadrightarrow \cdots \twoheadrightarrow \gl^{(-1)}=0.$$
We get a tower of Kan fibrations after applying the functor $\gamma_\bullet(-)$, whence a $\varprojlim^1$-sequence of homotopy groups (cf. \cite[Theorem IX.3.1]{BK}),
$$\xymatrix{{*} \ar[r] &  \varprojlim^1 \pi_{n+1}(\gl^{(r)}) \ar[r] & \pi_n(\gl) \ar[r] & \varprojlim \pi_n(\gl^{(r)}) \ar[r] & {*}},$$
where $\pi_n(\gl) = \pi_n(\gamma_\bullet(\gl),0) = \pi_n(\varprojlim \gamma_\bullet(\gl^{(r)}),0)$. There is also a $\varprojlim^1$-sequence associated to the tower of surjections of chain complexes,
$$\xymatrix{0 \ar[r] &  \varprojlim^1 \HH_{n+1}(\gl^{(r)}) \ar[r] & \HH_n(\gl) \ar[r] & \varprojlim \HH_n(\gl^{(r)}) \ar[r] & 0.}$$

\begin{proposition} \label{prop:lim^1}
Let $\gl$ be a complete $L_\infty$-algebra. The diagram
$$\xymatrix{0 \ar[r] &  \varprojlim^1 \HH_{n+1}(\gl^{(r)}) \ar[r] \ar[d]^-{(-1)^n\varprojlim^1 B_{n+1}} & \HH_n(\gl) \ar[r]^-{f^H} \ar[d]^-{B_n} & \varprojlim \HH_n(\gl^{(r)}) \ar[r] \ar[d]^-{\varprojlim B_n} & 0 \\
0 \ar[r] &  \varprojlim^1 \pi_{n+2}(\gl^{(r)}) \ar[r] & \pi_{n+1}(\gl) \ar[r]^-{f^\pi} & \varprojlim \pi_{n+1}(\gl^{(r)}) \ar[r] & 0}$$
is commutative for all $n\geq 0$.
\end{proposition}

\begin{proof}
It is easy to check that the right square commutes. To show that the left square commutes, we must first recall how the kernels of the maps $f^H$ and $f^\pi$ are identified with the respective $\limo$-groups.

First we recall the definition of $\limo$ for groups, cf.~\cite[IX.\S2]{BK}. Given a tower of groups,
$$\cdots \rightarrow G_r\stackrel{p}{\rightarrow} G_{r-1} \rightarrow \cdots \rightarrow G_{-1} = *,$$
$\limo G_r$ is defined as the set of equivalence classes
$$\limo G_r = \prod_r G_r / \sim$$
where $(x_r)_r \sim (y_r)_r$ if there is a sequence $(g_r)_r$ such that
$$y_r = g_r x_r (pg_{r+1})^{-1},\quad \mbox{for all $r\geq 0$}.$$

Next, let $[(\alpha_r)_r]$ be an element in the kernel of $f^H\colon \HH_k(\gl) \rightarrow \varprojlim \HH_k(\gl^{(r)})$. This means that each $\alpha_r\in \gl_n^{(r)}$ is a boundary; say $\delta \beta_r = \alpha_r$. Since $p\alpha_{r+1} = \alpha_r$ for all $r$, the element $p\beta_{r+1}-\beta_r$ is a cycle. The identification
$$\xymatrix{\ker f^H \ar[r]^-\cong & {\varprojlim}^1 \HH_{n+1}(\gl^{(r)})}$$
is effected by sending the class $[(\alpha_r)_r]$ to the equivalence class represented by the sequence $([p \beta_{r+1}-\beta_r])_r \in \prod_r \HH_{n+1}(\gl^{(r)})$.

The identification
$$\xymatrix{\ker f^\pi \ar[r]^-\cong & \limo \pi_{n+2}(\gl^{(r)})}$$
goes as follows (cf. \cite[IX.\S3]{BK}). Given an element $[(a_r)_r]$ in the kernel of $f^\pi$, each $[a_r]\in \pi_{n+2}(\gl^{(r)})$ is trivial, so there are $(n+2)$-simplices $b_r\in \gamma_{n+2}(\gl^{(r)})$ with simplicial boundary
$$\partial b_r = (a_r,0,\ldots,0).$$
Choose a filler $c_r\in \gamma_{n+3}(\gl^{(r)})$ for the horn
$$(pb_{r+1},b_r,-,0,\ldots,0) \colon \Lambda^2[n+3] \rightarrow \gamma_\bullet(\gl).$$
Then $\partial(\partial_2 c_r) = (0,\ldots,0)$, so $\partial_2 c_r$ represents a homotopy class in $\pi_{n+2}(\gl^{(r)})$. The element $[(a_r)_r]$ is sent to the equivalence class in $\limo \pi_{n+2}(\gl^{(r)})$ represented by the sequence
$$([\partial_2 c_r])_r \in \prod_r \pi_{n+2}(\gl^{(r)}).$$

Let $[(\alpha_r)_r]$ be an element in the kernel of $f^H$. The map $B_n\colon \HH_n(\gl)\rightarrow \pi_{n+1}(\gl)$ sends $[(\alpha_r)_r]$ to the class $[(a_r)_r]$ where $a_r := \alpha_r\tensor \omega^{n+1}$. By commutativity of the right square this class belongs to the kernel of $f^\pi$. To prove commutativity of the left square, it suffices to prove that a filler $c_r$ may be chosen such that
\begin{equation} \label{eq:c}
[\partial_2 c_r] = (-1)^n [(p\beta_{r+1}-\beta_r)\tensor \omega^{n+2}]\in \pi_{n+2}(\gl^{(r)}),\quad \mbox{for all $r\geq 0$}.
\end{equation}
To this end first observe that, in the notation above, we may choose
\begin{align*}
b_r & = (d+\delta)(\beta_r\tensor \omega_{\hat{0}}) \\
& = \alpha_r\tensor \omega_{\hat{0}} + (-1)^{n+1} \beta_r\tensor \omega^{n+2}.
\end{align*}
Indeed, $\partial b_r = (a_r,0,\ldots,0)$ and $b_r\in \gamma_{n+2}(\gl^{(r)})$. Next, consider the following element of $(\gl^{(r)}\tensor \Omega_{n+3})_{-1}$:
\begin{align*}
\lambda_0 & = (d+\delta)((-1)^n \alpha_r \tensor \omega_{3\ldots n+3} -\beta_r\tensor \omega_{03\ldots n+3} - p\beta_{r+1}\tensor \omega_{13\ldots n+3}) \\
& = \alpha_r\tensor \omega_{2\ldots n+3} + (-1)^{n+1} p\beta_{r+1}\tensor \omega_{1\ldots n+3} \\
& \quad + (-1)^{n+1} \beta_r\tensor \omega_{02\ldots n+3} + (-1)^n (p\beta_{r+1} - \beta_r)\tensor \omega_{013\ldots n+3}.
\end{align*}
It has simplicial boundary
$$\partial \lambda_0 = (pb_{r+1},b_r,(-1)^n(p\beta_{r+1}-\beta_n)\tensor \omega^{n+2},0,\ldots,0).$$
Clearly, $(d+\delta) \lambda_0 = 0$. If $n>1$, then for degree reasons $[\lambda_0^{\wedge \ell}] = 0$ for any $\ell \geq 2$. This is also true for $k = 1$ by direct calculation. Thus, if $n>0$ then $\lambda_0$ is a Maurer-Cartan element and we may choose $c_r =\lambda_0 \in \gamma_{n+3}(\gl^{(r)})$ as our filler. In this case we are done because $\partial_2 \lambda_0 = (-1)^n(p\beta_{r+1}-\beta_r)\tensor \omega^{n+2}$, so that \eqref{eq:c} is fulfilled already before passing to homotopy classes.

The case $k = 0$ requires a little more care. Then we have that
$$\lambda_ 0 = \alpha_r\tensor \omega_{23} - p\beta_{r+1} \tensor \omega_{123} - \beta_r\tensor \omega_{023} + (p\beta_{r+1}-\beta_r)\tensor \omega_{013}.$$
One verifies easily that $[\lambda_0^{\wedge \ell}] = 0$ for $\ell \geq 3$, but
$$\frac{1}{2} [\lambda_0,\lambda_0] = - [\alpha_r,p\beta_{r+1}-\beta_r]\tensor \frac{1}{3}t_3 \omega_{0123},$$
so $\lambda_0$ is not necessarily a Maurer-Cartan element. However, by \cite[Lemma 5.3]{Getzler} there is a unique element $\lambda \in \gamma_3(\gl^{(r)})$ such that $\epsilon_3^2 \lambda = 0$ and $P_3 R_3^2\lambda = \lambda_0$. Moreover, this element has the property that $\partial_i(\lambda) = \partial_i(\lambda_0)$ for $i\ne 2$. Therefore, we may choose $c_r:=\lambda$ as a filler, and it remains to verify that \eqref{eq:c} holds.

The element is obtained as a limit $\lambda  = \lim_{k\rightarrow \infty} \lambda_k$, where $\lambda_k$ is defined by the iterative formula
$$\lambda_k = \lambda_0 - \sum_{\ell \geq 2} \frac{1}{\ell !} F [\lambda_{k-1}^{\wedge \ell}],\quad k\geq 1,$$
for a certain operator\footnote{$F = P_3h_3^2 +s_3$ in the notation of \cite{Getzler}.} $F\colon \Omega_3^* \rightarrow \Omega_3^{*-1}$. In the first iteration, a calculation yields
$$\lambda_1 = \lambda_0 - [\alpha_r,p\beta_{r+1}-\beta_r]\tensor F(\frac{1}{3}t_3\omega_{0123}).$$
Observe that since $\alpha_r = \delta \beta_r$ and $p\beta_{r+1} - \beta_r$ is a cycle, we have that
$$[\alpha_r,p\beta_{r+1}-\beta_r] = \delta [\beta_r,p\beta_{r+1} -\beta_r],$$
so $\lambda_0 - \lambda_1$ is a $\delta$-boundary. Moreover, $\lambda_0-\lambda_1 \in \gl_1^{(r)}\tensor \Omega_3^2$. One checks by induction that the same is true for all $k$:
\begin{itemize}
\item $\lambda_0 -\lambda_k$ is a $\delta$-boundary.
\item $\lambda_0 - \lambda_k\in \gl_1^{(r)} \tensor \Omega_3^2$.
\end{itemize}
This implies that $\partial_2(\lambda_0)-\partial_2(\lambda)$ is a $\delta$-boundary in $\gl_1^{(r)}\tensor \Omega_2^2$. Now, $\partial_2(\lambda_0) = (p\beta_{r+1}-\beta_r)\tensor \omega^2 \in \gl_1^{(r)}\tensor \Omega_2^2$. Therefore $\partial_2(\lambda)\in \gl_1^{(n)}\tensor \Omega_2^2$. Since we also have that $\partial_2(\lambda)\in\gamma_2(\gl^{(r)})$, this implies that $\partial_2(\lambda) = \xi\tensor \omega^2$ for some cycle $\xi\in \gl_1^{(r)}$. Moreover, the elements $\xi$ and $p\beta_{r+1}-\beta_r$ differ by a boundary, whence
$$[\partial_2(\lambda)] = [(p\beta_{r+1}-\beta_r)\tensor \omega^2] \in \pi_2(\gl^{(r)}),$$
by well-definedness of the map $B_1\colon \HH_1(\gl^{(r)})\rightarrow \pi_2(\gl^{(r)})$. Thus \eqref{eq:c} is satisfied for $c_r = \lambda$, and this finishes the proof.
\end{proof}

\section{Application: rational models for mapping spaces} \label{sec:bs}
In this section, we show that the space of maps into a $\QQ$-local space is homotopy equivalent to the Deligne-Getzler $\infty$-groupoid of a complete $L_\infty$-algebra $A\htensor L$, where $A$ is a cdga model for the source and $L$ is an $L_\infty$-algebra model for the target.

As is well-known, the simplicial cdga $\Omega_\bullet$ gives rise to a (contravariant) adjunction between simplicial sets and cdgas,
\begin{equation} \label{eq:sadj}
\adjunction{\mathsf{sSet}}{\mathsf{CDGA}_\QQ^{op},}{\Omega}{\langle - \rangle}
\end{equation}
$$\Omega(X) = \Hom_{sSet}(X,\Omega_\bullet),\quad \langle B \rangle = \Hom_{cdga}(B,\Omega_\bullet).$$
The cdga $\Omega(X)$ is the Sullivan-deRham algebra of polynomial differential forms on $X$, and $\langle B \rangle$ is the \emph{spatial realization} of $B$. It is a fundamental result in rational homotopy theory that the adjunction induces an equivalence between the homotopy categories of nilpotent rational spaces of finite $\QQ$-type and minimal Sullivan algebras of finite type, see \cite{BG,Sullivan}. Under the correspondence in Theorem \ref{thm:sullivan}, the spatial realization corresponds to the nerve:

\begin{proposition} \label{prop:spatial}
Let $L$ be a non-negatively graded degree-wise nilpotent $L_\infty$-algebra of finite type and let $A$ be a cdga. If $A$ or $L$ is bounded, then there is a natural isomorphism
$$\MC(A\tensor L) \cong \Hom_{cdga}(C^*(L),A).$$
In particular, the nerve of $L$ is isomorphic to the spatial realization of its associated Sullivan algebra: $\MC_\bullet(L) \cong \langle C^*(L)\rangle$.
\end{proposition}

\begin{proof}
The underlying graded commutative algebra of $C^*(L)$ is free on $L[1]^\vee$. Thus, given an element $\tau$ in $A\tensor L$ of degree $-1$, we can define a morphism of graded algebras $f_\tau \colon C^*(L) \rightarrow A$ by $f_\tau(\xi) = (1\tensor \xi)(\tau)$. If $\tau$ is a Maurer-Cartan element, then $f_\tau$ commutes with differentials. We leave the rest of the proof to the reader.
\end{proof}

The category of cdgas is enriched in simplicial sets via
$$\Map_{cdga}(B,A) = \Hom_{cdga}(B,A\tensor \Omega_\bullet),$$
see \cite[\S5]{BG}, and the category of simplicial sets is enriched in itself via
$$\Map_{sSet}(X,Y) = \Hom_{sSet}(X\times \Delta[\bullet],Y).$$
It is natural to ask to what extent the adjunction \eqref{eq:sadj} is compatible with the simplicial enrichments. This has been answered by Brown and Szczarba.

\begin{theorem}[{Brown-Szczarba \cite[Theorem 2.20]{BS}}] \label{thm:bs}
Let $X$ be a connected simplicial set and let $B$ be a finite type Sullivan algebra. There is a natural homotopy equivalence of Kan complexes
$$BS\colon\Map_{cdga}(B,\Omega(X)) \stackrel{\simeq}{\rightarrow} \Map_{sSet}(X,\langle B \rangle).$$
Furthermore, the functor $\Map_{cdga}(B,-)$ takes quasi-isomorphisms between commutative differential graded algebras to homotopy equivalences between Kan complexes.
\end{theorem}

For completeness, we offer a short proof using nerves of $L_\infty$-algebras below. The following is a consequence.

\begin{theorem} \label{thm:map tensor}
Let $X$ be a connected simplicial set and let $Y$ be a nilpotent space of finite $\QQ$-type. If $L$ is a degree-wise nilpotent $L_\infty$-algebra model for $Y$ of finite type and $A$ is a cdga model for $X$, then there is a homotopy equivalence of Kan complexes
$$\Map(X,Y_\QQ) \simeq \gamma_\bullet(A\hat{\tensor} L),$$
where $A\htensor L$ is completed with respect to the filtration $A\tensor L_{\geq r}$.
\end{theorem}

\begin{proof}
By the Sullivan-deRham localization theorem \cite[\S11.2]{BG} we may take the spatial realization $Y_\QQ = \langle C^*(L) \rangle$ as a $\QQ$-localization of $Y$. The Brown-Szczarba theorem applied to $B = C^*(L)$ yields a homotopy equivalence between $\Map(X,Y_\QQ)$ and $\Map_{cdga}(C^*(L),A)$. We may write $L$ as the inverse limit of the finite dimensional nilpotent quotients $L/F^rL$, where $F^rL = L_{\geq r-1}$. It follows that
$$\Map_{cdga}(C^*(L),A) = \varprojlim \Hom_{cdga}(C^*(L/F^rL),A\tensor \Omega_\bullet) \cong \varprojlim \MC_\bullet(A\tensor L/F^rL),$$
where we have used Proposition \ref{prop:spatial} in the last step (we do not assume that $A$ is bounded, so we must pass to the finite dimensional quotients $L/F^rL$ before we can apply the proposition). The latter simplicial set is isomorphic to the nerve of the complete $L_\infty$-algebra $A\htensor L = \varprojlim A\htensor (L/F^r L)$. By Proposition \ref{prop:mcgamma}, this is homotopy equivalent to $\gamma_\bullet(A\hat{\tensor} L)$.
\end{proof}

\begin{remark}
Note that we do not assume that $X$ is finite in Theorem \ref{thm:map tensor}. This is the reason we need to extend the definition of $\gamma_\bullet$ to complete $L_\infty$-algebras. However, if the models $A$ and $L$ can be chosen so that either $A$ or $L$ is bounded, then the completed tensor product $A\hat{\tensor} L$ is isomorphic to the ordinary tensor product $A\tensor L$ and it is degree-wise nilpotent. Furthermore, note that if $A$ is the dual of a dg coalgebra $C$, which happens for instance if $A$ is of finite type, then $A\hat{\tensor} L \cong \Hom(C,L)$.
\end{remark}

\subsection*{Proof of Theorem \ref{thm:bs}}
We now embark on a proof of Theorem \ref{thm:bs}. The theorem is reformulated in the language of nerves of $L_\infty$-algebras as Theorem \ref{thm:map} below.

\begin{proposition} \label{prop:mu}
Let $X$ be a simplicial set and $L$ a non-negatively graded complete $L_\infty$-algebra. There is a natural map
$$\mu\colon\MC(\Omega(X)\hat{\tensor} L)\rightarrow \Hom_{sSet}(X,\MC_\bullet(L))$$
which is an isomorphism if $X$ is finite or if each $L/F^rL$ is finite dimensional.
\end{proposition}

\begin{proof}
We define the map for nilpotent $L$ first. Given $\tau\in \MC(\Omega(X)\tensor L)$, define a simplicial map $f\colon X\rightarrow \MC_\bullet(L)$ as follows. For an $n$-simplex $x\colon \Delta[n]\rightarrow X$, we let $f(x)\in \MC_n(L)$ be the image of $\tau$ under the map
$$x^*\colon \MC(\Omega(X)\tensor L)\rightarrow \MC(\Omega_n\tensor L).$$
It is straightforward to check that $f$ is a simplicial map, and we set $\mu(\tau) = f$. The map $\mu$ is evidently an isomorphism for $X=\Delta[n]$. Since $-\tensor L$ commutes with finite limits, the functor $\MC(\Omega(-)\tensor L)$ takes finite colimits to limits. The functor $\Hom_{sSet}(-,\MC_\bullet(L))$ preserves all limits, so it follows that $\mu$ is an isomorphism for finite $X$. On the other hand, if $L$ is finite dimensional then $-\tensor L$ commutes with all limits and in this case $\mu$ is an isomorphism for arbitrary $X$. Finally, for complete $L$ the map $\mu$ is defined as the map induced on inverse limits
$$\varprojlim \MC(\Omega(X)\tensor L/F^r L) \rightarrow \varprojlim \Hom_{sSet}(X,\MC_\bullet(L/F^rL)),$$
and by the above this is an isomorphism if $X$ is finite or if each $L/F^rL$ is finite dimensional.
\end{proof}

\begin{theorem} \label{thm:map}
Let $X$ be a simplicial set and $L$ a non-negatively graded complete $L_\infty$-algebra. There is a natural homotopy equivalence of Kan complexes
$$\varphi\colon\MC_\bullet(\Omega(X)\hat{\tensor} L)\rightarrow \Map(X,\MC_\bullet(L)).$$
Furthermore, the functor $\MC_\bullet(-\hat{\tensor} L)\colon \mathsf{CDGA} \rightarrow \mathsf{sSet}$ takes quasi-isomorphisms to homotopy equivalences between Kan complexes.
\end{theorem}

\begin{remark}
There is a commutative diagram
$$\xymatrix{\MC_\bullet(\Omega(X)\hat{\tensor} L) \ar[r]^-\varphi \ar[d]^-\cong & \Map(X,\MC_\bullet(L)) \ar[d]^-\cong \\ \Map_{cdga}(C^*(L),\Omega(X)) \ar[r]^-{BS} & \Map(X,\langle C^*(L) \rangle).}$$
Therefore, Theorem \ref{thm:map} may be viewed as a reformulation of the Brown-Szczarba theorem.
\end{remark}

\begin{proof}
We define the map $\varphi$ for nilpotent $L$ first. On $n$-simplices, the map $\varphi_n$ is defined as the composite
\begin{align*}
\MC_n(\Omega(X)\tensor L)
& = \MC(\Omega(\Delta[n])\tensor \Omega(X) \tensor L) \\
& \stackrel{\pi}\rightarrow \MC(\Omega(\Delta[n]\times X)\tensor L) \\
& \stackrel{\mu}{\rightarrow} \Hom_{sSet}(\Delta[n]\times X,\MC_\bullet(L)) \\
& = \Map(X,\MC_\bullet(L))_n,
\end{align*}
where $\pi$ is induced by the natural morphism of cdgas $\Omega(\Delta[n])\tensor \Omega(X)\rightarrow \Omega(\Delta[n]\times X)$, and $\mu$ is the map described in Proposition \ref{prop:mu}. By naturality of all maps involved, $\varphi$ respects the simplicial structure. For complete $L$, we define $\varphi$ to be the induced map on inverse limits
$$\varprojlim \MC_\bullet(\Omega(X)\tensor L/F^rL)\rightarrow \varprojlim \Map(X,\MC_\bullet(L/F^rL)).$$

The proof that $\varphi$ is a weak equivalence is by reduction to the abelian case. For abelian $L$, the map $\varphi$ is equivalent to the map between Dold-Kan constructions
$$\Gamma_\bullet(\Omega(X)\tensor L[1])\rightarrow \Map(X,\Gamma_\bullet(L[1])).$$
But $\Gamma_\bullet(L[1])$ is a product of rational Eilenberg-MacLane spaces, and by standard arguments the homotopy groups of the right hand side are given by $\HH^*(X;\QQ)\tensor \HH_*(L)[1]$, and $\varphi$ is a weak equivalence in this case. Next, we note that both source and target of $\varphi$ take central extensions of $L_\infty$-algebras to principal fibrations, so we can induct on a composition series for $L$ to conclude that $\varphi$ is a weak equivalence for all nilpotent $L_\infty$-algebras $L$. For complete $L$, we get a levelwise weak equivalence of towers of fibrations
$$\MC_\bullet(\Omega(X)\tensor L/F^rL)\stackrel{\sim}{\rightarrow} \Map(X,\MC_\bullet(L/F^rL)).$$
Hence, we get a weak equivalence upon passing to inverse limits. This finishes the proof of the first statement. The proof that $\MC_\bullet(-\hat{\tensor} L)$ takes quasi-isomorphisms to weak equivalences is proved by a similar reduction to the abelian case.
\end{proof}

\section{Examples}

\subsection*{Finite rational cohomology or homotopy}
If $\HH^*(X;\QQ)$ or $\pi_*(Y)\tensor \QQ$ is finite dimensional, then for a fixed map $g\colon X\rightarrow Y$, composition with the $\QQ$-localization map $r\colon Y\rightarrow Y_\QQ$ induces a rational homotopy equivalence
$$\xymatrix{\Map(X,Y;g) \ar[r]^-{\sim_\QQ} & \Map(X,Y_\QQ;rg).}$$
If $X$ has finite dimensional rational cohomology, we can find a finite dimensional cdga model $A$ (see e.g.~\cite[Example 6, p.146]{FHT-RHT}). Likewise, if $\pi_*(Y)\tensor \QQ$ is finite dimensional, then we can find a finite dimensional nilpotent $L_\infty$-algebra model $L$ for $Y$. In either of these situations, $A\hat{\tensor} L = A\tensor L$, and given a Maurer-Cartan element $\tau\in \MC(A\tensor L)$ that represents the map $rg\colon X\rightarrow Y_\QQ$, the truncated and twisted $L_\infty$-algebra $(A\tensor L^\tau)_{\geq 0}$ is an $L_\infty$-algebra model for the component $\Map(X,Y;g)$. Equivalently, the Chevalley-Eilenberg construction $C^*((A\tensor L^\tau)_{\geq 0})$ is a (not necessarily minimal) Sullivan model for $\Map(X,Y;g)$. In particular, we have an isomorphism of graded Lie algebras
$$\pi_{*+1}(\Map(X,Y),g)\tensor \QQ \cong \HH_*(A\tensor L^\tau),\quad *\geq 0.$$
Here, $\pi_1\tensor \QQ$ is interpreted as the Malcev completion of the nilpotent group $\pi_1$ \cite[\S A3]{Quillen}, and $\HH_0(A\tensor L^\tau)$ is given the Campbell-Hausdorff group structure.

\subsection*{Inclusions of complex projective spaces.}
The rational homotopy of maps into projective spaces has been studied in \cite{MR}. The following example shows that the calculations are drastically simplified by using Theorem \ref{thm:main3}.

Consider the standard inclusion $i\colon \CP{n}\rightarrow \CP{m}$ where $m\geq n\geq 1$. As a cdga model for $\CP{n}$ we may choose the cohomology
$$A = \HH^*(\CP{n};\QQ) = \QQ[x]/(x^{n+1}),\quad |x| = 2.$$
A minimal $L_\infty$-algebra model for $\CP{m}$ is given by
$$L =  \pi_*(\Omega \CP{m})\tensor \QQ = \langle \alpha,\beta \rangle,\quad |\alpha| = 1, \quad |\beta| = 2m,$$
where the only non-vanishing bracket is
$$\frac{1}{(m+1)!}[\alpha^{\wedge m+1}] = \beta.$$
The twisting cochain $\tau = x\tensor \alpha \in A\tensor L$ represents the inclusion $i\colon \CP{n}\rightarrow \CP{m}$. Thus, an $L_\infty$-model for the component $\Map(\CP{n},\CP{m};i)$ is given by the finite dimensional $L_\infty$-algebra $\gl =  A\tensor L_{\geq 0}^\tau$, and the Chevalley-Eilenberg construction $C^*(\gl)$ is a Sullivan model. A basis for $\gl$ is given by
$$
\begin{array}{ccccc}
2m & 2m-2 & \cdots & 2m-2n & 1 \\
1\tensor \beta & x\tensor \beta & \cdots & x^n\tensor \beta & 1\tensor \alpha
\end{array}
$$
The $L_\infty$-algebra structure is described by
$$\frac{1}{r!}[(1\tensor \alpha)^{\wedge r}]_\tau = \binom{m+1}{r} x^{m+1-r}\tensor \beta.$$
Since $\beta$ is central in $L$, it follows that the elements $x^r \tensor \beta$ are central in $\gl$, so the only possible non-zero brackets are described by the above. Note that we get zero if $r\leq m-n$ since $x^{n+1} = 0$. In particular, the differential is zero if $m>n$, and in this case the Chevalley-Eilenberg construction $C^*(\gl)$ is a \emph{minimal} Sullivan model for $\Map(\CP{n},\CP{m};i)$. It has the following description:
$$\Lambda(z,w_{m-n},w_{m-n+1},\ldots,w_m),\quad dz = 0,\quad dw_r = z^{r+1},\quad |z| = 2,\quad |w_r| = 2r+1.$$
On the other hand, if $m=n$, then we have a non-zero differential $[1\tensor \alpha]_\tau = (n+1)x^n\tensor \beta$, and hence $\pi_{*}(\Map(\CP{n},\CP{n}),1_{\CP{n}})\tensor \QQ$ is concentrated in odd degrees (as predicted by Halperin's conjecture, see below) with precisely one basis element each in the degrees $3,5,\ldots,2n+1$. The minimal Sullivan model for $\Map(\CP{n},\CP{n};1_{\CP{n}})$ is therefore an exterior algebra on the dual of the rational homotopy with zero differential:
$$\Lambda(x_3,x_5,\ldots,x_{2n+1}),\quad dx_j = 0,\quad |x_j| = j.$$
In other words, $\Map(\CP{n},\CP{n};1_{\CP{n}})$ is rationally homotopy equivalent to the product $S^3\times S^5 \times \cdots \times S^{2n+1}$.

\subsection*{On the Halperin conjecture}
Let $X$ be an \emph{$F_0$-space}, i.e., a simply connected space with evenly graded rational cohomology such that both $\dim_\QQ \HH^*(X;\QQ)<\infty$ and $\dim_\QQ \pi_*(X)\tensor \QQ < \infty$. Then $X$ is formal and the cohomology algebra admits a presentation
$$\HH^*(X;\QQ) = \QQ[x_1,\ldots,x_n]/(f_1,\ldots,f_n),$$
where $x_1,\ldots,x_n$ are evenly graded generators and $f_1,\ldots,f_n$ is a regular sequence \cite{Halperin}. \emph{Halperin's conjecture} says that for such spaces $X$, the component $\aut_1(X)$ of the space of homotopy self-equivalences of $X$ that contains the identity map is rationally homotopy equivalent to a product of odd dimensional spheres. Equivalently, the rational homotopy groups $\pi_*(\aut X,1_X)\tensor \QQ$ are concentrated in odd degrees. Using Theorem \ref{thm:main3}, we get a new simpler proof of the following result.

\begin{theorem}(Meier \cite{M})
Let $X$ be an $F_0$-space. Then the Halperin conjecture holds for $X$ if and only if the cohomology algebra $\HH^*(X;\QQ)$ admits no derivations of negative degree.
\end{theorem}

\begin{proof}
A minimal model for $X$ is given by
$$\QQ[x_1,\ldots,x_n]\tensor \Lambda(y_1,\ldots,y_n),\quad dx_i = 0, \quad dy_i = f_i.$$
Here $y_i$ is a generator of odd degree $|y_i| = |f_i|-1$. The dual $L_\infty$-algebra $L$ is given by
$$L = L_{odd} \oplus L_{even} = \langle \alpha_1,\ldots,\alpha_n\rangle \oplus \langle\beta_1,\ldots,\beta_n \rangle,$$
where $|\alpha_i| = |x_i|-1$ are odd  and $|\beta_i| = |y_i|-1$ even. We have an isomorphism of rational vector spaces for every $n\geq 0$,
$$\pi_{n+1}(\aut X,1_X)\tensor \QQ = \HH_n(A\tensor L,D^\pi),\quad D^\pi(\xi) = \sum_{\ell \geq 2} \frac{1}{\ell!} [\pi^{\wedge \ell},\xi],$$
where the Maurer-Cartan element $\pi$ is given by $\pi = \sum_i x_i\tensor \alpha_i + y_i\tensor \beta_i.$
Hence, the Halperin conjecture holds for $X$ if and only if
$$\HH_{n}(A\tensor L,D^\pi) = 0,\quad \mbox{for all odd $n>0$.}$$
Since $L_{even}$ is central in $L$, we have that $D^\pi(A\tensor L_{even}) = 0$. Thus,
$$\HH_{odd}(A\tensor L,D^\pi) = \ker (D^\pi \colon A\tensor L_{odd}\rightarrow A\tensor L_{even}).$$
The differential $D^\pi$ is $A$-linear, and a calculation yields
$$D^\pi(1\tensor \alpha_i) = \sum_{j=1}^n \frac{\partial f_j}{\partial x_i}\tensor \beta_j.$$
Therefore, an element $\sum_i p_i\tensor \alpha_i\in A\tensor L_{odd}$ belongs to the kernel of $D^\pi$ if and only if
$$\sum_{i=1}^n p_i \frac{\partial f_j}{\partial x_i} = 0,\quad 1\leq j\leq n.$$
But this is true if and only if $\sum_i p_i\frac{\partial}{\partial x_i}$ defines a derivation of $A$. Thus, we obtain an isomorphism
$$\ker(D^\pi\colon A\tensor L_{odd} \rightarrow A\tensor L_{even}) \cong Der A, \quad \sum_{i=1}^n p_i\tensor \alpha_i \mapsto \sum_{i=1}^n p_i\frac{\partial}{\partial x_i}.$$
This shows that $\HH_n(A\tensor L,D^\pi) = 0$ for all odd $n>0$ if and only if $A$ admits no negative (in cohomological grading) derivations.
\end{proof}

\subsection*{Homotopy automorphisms of formal and coformal spaces}
In \cite{KS}, the following characterization of spaces that are both formal and coformal was established.
\begin{theorem}[Berglund {\cite{KS}}] \label{thm:koszul space}
The following are equivalent for a connected nilpotent space $X$ of finite $\QQ$-type:
\begin{enumerate}
\item $X$ is both formal and coformal.
\item $X$ is formal and $\HH^*(X;\QQ)$ is a Koszul algebra.
\item $X$ is coformal and $\pi_*(\Omega X)\tensor \QQ$ is a Koszul Lie algebra.
\end{enumerate}
In this situation, homotopy is Koszul dual to cohomology in the sense that
$$\pi_*(\Omega X)\tensor \QQ \cong \HH^*(X;\QQ)^{!_{\Lie}}.$$
\end{theorem}

That the cohomology $\HH^*(X;\QQ)$ is a Koszul algebra means that it is generated by elements $x_i$ modulo quadratic relations $\sum_i c_{ij}x_ix_j = 0$ such that $\Tor_{s,t}^A(\QQ,\QQ) = 0$ for $s\ne t$, where the extra grading on $\Tor$ is induced by wordlength in the generators $x_i$. That $\pi_*(\Omega X)\tensor \QQ \cong \HH^*(X;\QQ)^{!_{\Lie}}$ means that, as a graded Lie algebra, $\pi_*(\Omega X)\tensor \QQ$ is generated by classes $\alpha_i$ dual to $x_i$ modulo the \emph{orthogonal relations}: a relation
$$\sum_{i,j} \lambda_{ij} [\alpha_i,\alpha_j] = 0$$
holds if and only if
$$\sum_{i,j} (-1)^{|x_i||\alpha_j|} c_{ij}\lambda_{ij} = 0$$
whenever the coefficients $c_{ij}$ represent a relation among the generators $x_i$.

The component of the mapping space $\Map(X,X)$ that contains a fixed homotopy self-equivalence is equal to the same component of $\aut(X)$, since any map homotopic to a homotopy equivalence is itself a homotopy equivalence. Moreover, $\pi_1(\aut(X),1_X)$ is an abelian group as $\aut(X)$ is a monoid. By combining Theorem \ref{thm:main3}\footnote{Alternatively, one could use the main theorem of \cite{BFM-tran} in this case.} and Theorem \ref{thm:koszul space}, we obtain the following theorem.

\begin{theorem}
Let $X$ be formal and coformal nilpotent space such that either $\dim_\QQ \HH^*(X;\QQ)<\infty$ or $\dim_\QQ \pi_*(X)\tensor \QQ<\infty$. Then there is a finite basis $x_1,\ldots,x_n$ for the indecomposables of the cohomology algebra $\HH^*(X;\QQ)$ and a dual basis $\alpha_1,\ldots,\alpha_n$ for the indecomposables of the homotopy Lie algebra $\pi_*(\Omega X)\tensor \QQ$. Setting $\kappa = x_1\tensor \alpha_1 + \ldots + x_n\tensor \alpha_n\in \HH^*(X;\QQ)\tensor \pi_*(\Omega X)$, the derivation $[\kappa,-]$ is a differential, and there are isomorphisms
$$\pi_{k+1}(\aut(X),1_X) \tensor \QQ \cong \HH_k(\HH^*(X;\QQ)\tensor \pi_*(\Omega X),[\kappa,-]),\quad k\geq 0.$$
\end{theorem}
In \cite{BM}, we use this result to calculate the rational homotopy groups of the space of homotopy self-equivalences of highly connected manifolds. Further applications are treated in the PhD thesis of Casper Guldberg \cite{Guldberg}.

\subsection*{Acknowledgments}
The author is grateful to Greg Lupton for discussions about the rational homotopy theory of mapping spaces, and to Ryszard Nest for discussions about Deligne groupoids. We also thank the anonymous referee for helpful comments and suggestions. This work was supported by the Danish National Research Foundation (DNRF) through the Centre for Symmetry and Deformation.


\begin{thebibliography}{99}
\bibitem{KS} A. Berglund, Koszul spaces, \emph{Trans. Amer. Math. Soc.} \textbf{366} (2014), no. 9, 4551--4569.

\bibitem{BM} A. Berglund, I. Madsen, Homological stability of diffeomorphism groups, \emph{Pure Appl. Math. Q.} \textbf{9} (2013), no. 1, 1--48. 

\bibitem{BL} J. Block, A. Lazarev, Andr\'e-Quillen cohomology and rational homotopy of function spaces, \emph{Adv. Math.} \textbf{193} (2005), no. 1, 18--39.

\bibitem{BG} A.K. Bousfield, V.K.A.M. Gugenheim, On $PL$ DeRham theory and rational homotopy type, \emph{Mem. Amer. Math. Soc.}, vol. 8, No. 179, Amer. Math. Soc., Providence, R.I., 1976.

\bibitem{BK} A.K. Bousfield, D.M. Kan, \emph{Homotopy limits, completions and localizations}. Lecture Notes in Mathematics, vol. 304. Springer-Verlag, Berlin-New York, 1972.

\bibitem{BPS} A.K Bousfield, C. Peterson, L. Smith, The rational homology of function spaces, \emph{Arch. Math. (Basel)} \textbf{52} (1989), no. 3, 275--283.

\bibitem{BS} E.H. Brown, R.H. Szczarba, Continuous cohomology and real homotopy type, \emph{Trans. Amer. Math. Soc.} \textbf{311} (1989), no. 1, 57--106.

\bibitem{BS2} E.H. Brown, R.H. Szczarba, On the rational homotopy type of function spaces, \emph{Trans. Amer. Math. Soc.} \textbf{349}, (1997), no. 12, 4931--4951.

\bibitem{BuM} U. Buijs, A. Murillo, The rational homotopy Lie algebra of function spaces, \emph{Comment. Math. Helv.} \textbf{83} (2008), no. 4, 723--739.

\bibitem{BFM-tran} U. Buijs, Y. F\'elix, A. Murillo, Lie models for the components of sections of a nilpotent fibration, \emph{Trans. Amer. Math. Soc.} \textbf{361} (2009), no. 10, 5601--5614.

\bibitem{BFM-jap} U. Buijs, Y. F\'elix, A. Murillo, $L_\infty$ models of based mapping spaces, \emph{J. Math. Soc. Japan} \textbf{63} (2011), no. 2, 503--524.

\bibitem{BFM-new} U. Buijs, Y. F\'elix, A. Murillo, $L_\infty$ rational homotopy of mapping spaces, \emph{Rev. Mat. Complut.} \textbf{26} (2013), no. 2, 573--588.

\bibitem{Dupont} J.L. Dupont, \emph{Curvature and characteristic classes}, Lecture Notes in Mathematics, vol. 640. Springer-Verlag, Berlin-New York, 1978.

\bibitem{FHT-RHT} Y. F\'elix, S. Halperin, J-C. Thomas, \emph{Rational Homotopy Theory}, Graduate texts in Mathematics 205, Springer, 2001.

\bibitem{FLS} Y. F\'elix, G. Lupton, S.B. Smith, The rational homotopy type of the space of self-equivalences of a fibration, \emph{Homology, Homotopy Appl.} \textbf{12} (2010), no. 2, 371--400.

\bibitem{Getzler2} E. Getzler, A Darboux theorem for Hamiltonian operators in the formal calculus of variations, \emph{Duke Math. J.} \textbf{111} (2002), no. 3, 535--560.

\bibitem{Getzler} E. Getzler, Lie theory for nilpotent $L_\infty$-algebras, \emph{Ann. of Math.} (2) {\bf 170} (2009), no. 1, 271--301.

\bibitem{GJ} P.G. Goerss, J.F. Jardine, \emph{Simplicial homotopy theory}, Progress in Mathematics, 174. Birkh\"auser Verlag, Basel, 1999.

\bibitem{GM} W. Goldman, J.J. Millson, The deformation theory of representations of fundamental groups of compact K\"ahler manifolds, \emph{Inst. Hautes \'Etudes Sci. Publ. Math.} No. 67 (1988), 43--96.

\bibitem{Guldberg} C. Guldberg, \emph{On homotopy automorphisms of Koszul spaces}, PhD thesis, University of Copenhagen, 2015.

\bibitem{Haefliger} A. Haefliger, Rational homotopy of the space of sections of a nilpotent bundle, \emph{Trans. Amer. Math. Soc.} \textbf{273} (1982), no. 2, 609--620.

\bibitem{Halperin} S. Halperin, Finiteness in the minimal models of Sullivan, \emph{Trans. Amer. Math. Soc.} \textbf{230} (1977), 173--199.

\bibitem{Henriques} A. Henriques, Integrating $L_\infty$-algebras, \emph{Compos. Math.} \textbf{144} (2008), no. 4, 1017--1045.

\bibitem{Hinich} V. Hinich, Descent of Deligne groupoids, \emph{Internat. Math. Res. Notices} (1997), no. 5, 223--239.

\bibitem{Lazarev} A. Lazarev, Maurer-Cartan moduli and models for function spaces, \emph{Adv. Math.} \textbf{235} (2013), 296--320. 

\bibitem{LS} G. Lupton, S.B. Smith, Whitehead products in function spaces: Quillen model formulae, \emph{J. Math. Soc. Japan} \textbf{62} (2010), no. 1, 49--81.

\bibitem{MR} J. M. M\o{}ller, M. Raussen, Rational homotopy of spaces of maps into spheres and complex projective spaces, \emph{Trans. Amer. Math. Soc.} \textbf{292} (1985), no. 2, 721--732. 

\bibitem{M} W. Meier, Rational universal fibrations and flag manifolds, \emph{Math. Ann.} \textbf{258} (1982), 329--340.

\bibitem{PS} S. Papadima, A. Suciu, Homotopy Lie algebras, lower central series and the Koszul property, \emph{Geom. Topol.} \textbf{8} (2004), 1079--1125.

\bibitem{Quillen} D. Quillen, Rational homotopy theory, \emph{Ann. of Math.} (2) {\bf 90} (1969) 205--295.

\bibitem{Serre} J-P. Serre, \emph{Lie algebras and Lie groups. Second edition.} Lecture Notes in Mathematics, 1500. Springer-Verlag, 1992.

\bibitem{Smith} S.B. Smith, The homotopy theory of function spaces: a survey, \emph{Homotopy theory of function spaces and related topics, 3--39, Contemp. Math.}, 519, Amer. Math. Soc., Providence, RI, 2010.

\bibitem{Sullivan} D. Sullivan, Infinitesimal computations in topology, \emph{Inst. Hautes \'Etudes Sci. Publ. Math.} No. 47 (1977), 269--331 (1978).
\end{thebibliography}
\end{document}